\documentclass{amsart}

\usepackage[T1]{fontenc}
\usepackage{enumerate, amsmath, amsfonts, amssymb, amsthm, comment, mathrsfs, wasysym, graphics, graphicx, xcolor, url, hyperref, hypcap, shuffle, xargs, multicol, overpic, pdflscape, multirow, hvfloat, minibox, accents, array, multido, xifthen, a4wide, ae, aecompl, blkarray, pifont, mathtools, etoolbox, dsfont}
\usepackage{marginnote}
\hypersetup{colorlinks=true, citecolor=darkblue, linkcolor=darkblue}
\usepackage[all]{xy}
\usepackage[bottom]{footmisc}
\usepackage{tikz}
\usetikzlibrary{trees, decorations, decorations.markings, shapes, arrows, matrix, calc, fit, intersections, patterns, angles}
\usepackage[external]{forest}
\graphicspath{{figures/}{figures/nodes/}}
\makeatletter\def\input@path{{figures/}}\makeatother
\usepackage{caption}
\usepackage[noabbrev,capitalise]{cleveref}
\usepackage[export]{adjustbox}
\usepackage{ulem}

\newtheorem{theorem}{Theorem}
\newtheorem{corollary}[theorem]{Corollary}
\newtheorem{proposition}[theorem]{Proposition}
\newtheorem{lemma}[theorem]{Lemma}
\newtheorem{conjecture}[theorem]{Conjecture}
\newtheorem*{theorem*}{Theorem}

\theoremstyle{definition}
\newtheorem{definition}[theorem]{Definition}
\newtheorem{example}[theorem]{Example}
\newtheorem{remark}[theorem]{Remark}

\crefname{notation}{Notation}{Notations}
\crefname{problem}{Problem}{Problems}
 
\renewcommand{\b}[1]{\boldsymbol{#1}} 
\newcommand{\f}[1]{\mathfrak{#1}} 

\newcommand{\set}[2]{\left\{ #1 \;\middle|\; #2 \right\}} 
\newcommand{\bigset}[2]{\big\{ #1 \;\big|\; #2 \big\}} 
\newcommand{\ssm}{\smallsetminus} 
\newcommand{\symdif}{\,\triangle\,} 
\newcommand{\eqdef}{\mbox{\,\raisebox{0.2ex}{\scriptsize\ensuremath{\mathrm:}}\ensuremath{=}\,}} 

\newcommand{\ie}{\textit{i.e.}~} 
\newcommand{\eg}{\textit{e.g.}~} 
\definecolor{darkblue}{rgb}{0,0,0.7} 
\definecolor{green}{RGB}{57,181,74} 
\definecolor{violet}{RGB}{147,39,143} 
\newcommand{\darkblue}{\color{darkblue}} 
\newcommand{\defn}[1]{\textsl{\darkblue #1}} 
\newcommand{\OEIS}[1]{\cite[{\rm \href{http://oeis.org/#1}{\texttt{#1}}}]{OEIS}}

\usepackage{todonotes}

\newcommand{\meet}{\wedge} 
\newcommand{\join}{\vee} 
\newcommand{\bigMeet}{\bigwedge} 
\newcommand{\bigJoin}{\bigvee} 
\newcommandx{\projDown}[1][1={}]{\smash{\pi_\downarrow^{#1}}} 
\newcommandx{\projUp}[1][1={}]{\smash{\pi^\uparrow_{#1}}} 
\newcommand{\con}{\mathrm{con}} 
\newcommandx{\JI}[1][1=L]{\mathcal{JI}(#1)} 
\newcommandx{\MI}[1][1=L]{\mathcal{MI}(#1)} 
\newcommandx{\UJI}[1][1=\equiv]{\mathcal{UJI}(#1)} 
\newcommandx{\UMI}[1][1=\equiv]{\mathcal{UMI}(#1)} 
\newcommand{\CJR}{\mathbf{cjr}} 
\newcommand{\CMR}{\mathbf{cmr}} 
\newcommandx{\CJC}[1][1=L]{\mathcal{CJC}(#1)} 
\newcommandx{\CMC}[1][1=L]{\mathcal{CMC}(#1)} 
\newcommandx{\CC}[1][1=L]{\mathcal{CC}(#1)} 
\newcommand{\row}{\mathrm{row}} 
\newcommand{\less}{\triangleleft} 
\newcommand{\lesscover}{\mathrel{\ooalign{$\less$\cr\hidewidth\hbox{$\cdot\mkern0.8mu$}\cr}}} 
\newcommand{\more}{\triangleright} 
\newcommand{\morecover}{\mathrel{\ooalign{$\more$\cr\hidewidth\hbox{$\cdot\mkern3mu$}\cr}}} 
\newcommandx{\upIdeal}[2][1=x, 2=]{\langle #1 \rangle\!_{#2}^\uparrow} 
\newcommandx{\downIdeal}[2][1=x, 2=]{\langle #1 \rangle\!^{#2}_\downarrow} 

\makeatletter
\def\l@part{\@tocline{1}{8pt}{0pc}{}{}}
\def\l@section{\@tocline{1}{4pt}{0pc}{}{}}
\makeatother
\let\oldtocpart=\tocpart
\renewcommand{\tocpart}[2]{\sc\large\oldtocpart{#1}{#2}}
\let\oldtocsection=\tocsection
\renewcommand{\tocsection}[2]{\bf\oldtocsection{#1}{#2}}
\let\oldtocsubsubsection=\tocsubsubsection
\renewcommand{\tocsubsubsection}[2]{\quad\oldtocsubsubsection{#1}{#2}}

\title{The canonical complex of the weak order}

\thanks{Partially supported by the French ANR grants CAPPS~17\,CE40\,0018 and CHARMS~19\,CE40\,0017.}

\author{Doriann Albertin}
\address[DA]{LIGM, Université Gustave Eiffel, CNRS, ESIEE Paris, F-77454 Marne-la-Vallée, France}
\email{doriann.albertin@u-pem.fr}
\urladdr{\url{https://doriann-albertin.github.io/site/}}

\author{Vincent Pilaud}
\address[VP]{CNRS \& LIX, \'Ecole Polytechnique, Palaiseau}
\email{vincent.pilaud@lix.polytechnique.fr}
\urladdr{\url{http://www.lix.polytechnique.fr/~pilaud/}}

\begin{document}

\begin{abstract}
We define and study the canonical complex of a finite semidistributive lattice~$L$.
It is the simplicial complex on the join or meet irreducible elements of~$L$ which encodes each interval of~$L$ by recording the canonical join representation of its bottom element and the canonical meet representation of its top element. 
This complex behaves properly with respect to lattice quotients of~$L$, in the sense that the canonical complex of a quotient of~$L$ is the subcomplex of the canonical complex of~$L$ induced by the join or meet irreducibles of~$L$ uncontracted in the quotient.
We then describe combinatorially the canonical complex of the weak order on permutations in terms of semi-crossing arc bidiagrams, formed by the superimposition of two non-crossing arc diagrams of N.~Reading.
We provide explicit direct bijections between the semi-crossing arc bidiagrams and the weak order interval posets of G.~Ch\^atel, V.~Pilaud and V.~Pons.
Finally, we provide an algorithm to describe the Kreweras maps in any lattice quotient of the weak order in terms of semi-crossing arc bidiagrams.
\end{abstract}

\vspace*{-1.3cm}
\maketitle
\vspace{-1.2cm}

\begin{figure}[h]
	\centerline{\includegraphics[scale=.8]{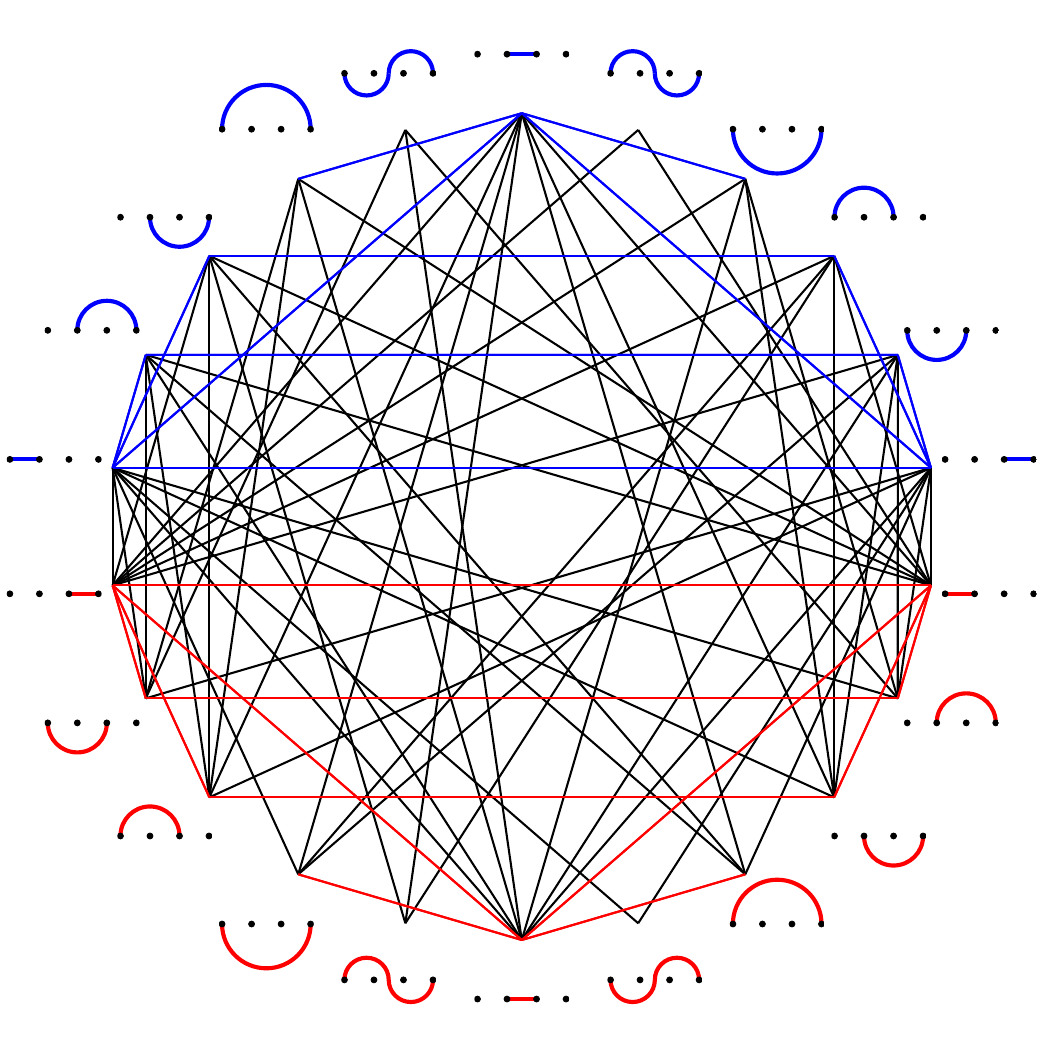}}
	\caption{The canonical complex of the weak order on $\mathfrak{S}_4$ labeled by arcs.}
	\label{fig:CCS4}
\end{figure}

\vspace{-2cm}

\newpage
\tableofcontents


\section{Introduction}
\label{sec:intro}

A finite lattice~$L$ is join semidistributive when any element admits a canonical join representation (see \eg \cite{FreeseNation} for a classical reference on lattices).
This enables us to define the canonical join complex of~$L$ \cite{Reading-arcDiagrams, Barnard}, whose vertices are the join irreducible elements of~$L$ and whose simplices are the canonical join representations in~$L$.
When $L$ is both join and meet semidistributive, it thus admits both a canonical join complex and a canonical meet complex which are actually isomorphic flag simplicial complexes \cite{Barnard}.

In the first part of this paper, we define the canonical complex of a finite semidistributive lattice~$L$, a larger flag simplicial complex where the canonical join complex and the canonical meet complex naturally live and interact.
More precisely, its vertex set is the disjoint union of the set of join irreducible elements of~$L$ with the set of meet irreducible elements of~$L$, and its simplices are the disjoint unions~$J \sqcup M$ of a canonical join representation~$J$ in~$L$ with a canonical meet representation~$M$ in~$L$ such that~$\bigJoin J \le \bigMeet M$.
In other words, each interval~$[x,y]$ in~$L$ contributes to a simplex of the canonical complex given by the disjoint union of the canonical join representation of~$x$ with the canonical meet representation of~$y$.
This provides a model for the intervals of~$L$ which is compatible with lattice quotients.
Namely, the canonical complex of a quotient~$L/{\equiv}$ is the subcomplex of the canonical complex of~$L$ induced by the join and meet irreducibles of~$L$ uncontracted by the congruence~$\equiv$.

In the second part of this paper, we study the combinatorics of the canonical complex of the weak order.
N.~Reading showed in~\cite{Reading-arcDiagrams} that join irreducible permutations correspond to certain arcs wiggling around the horizontal axis, and that canonical join representations of permutations correspond to non-crossing arc diagrams.
We show that the elements of the canonical complex can be interpreted as semi-crossing arc bidiagrams, defined as pairs~$\delta_\join \sqcup \delta_\meet$ of non-crossing arc diagrams where only certain types of crossings are allowed between an arc of~$\delta_\join$ and an arc of~$\delta_\meet$.
It thus follows that the canonical complex of any quotient of the weak order is isomorphic to a subcomplex of the semi-crossing complex induced by arcs contained in an upper ideal of the subarc order.
We then provide explicit direct bijections between the semi-crossing arc bidiagrams and the weak order interval posets of G.~Ch\^atel, V.~Pilaud and V.~Pons~\cite{ChatelPilaudPons}, which are both in bijection with the intervals of the weak order.
Finally, we provide an algorithm to describe the Kreweras maps in any lattice quotient of the weak order in terms of semi-crossing arc bidiagrams, generalizing the classical Kreweras complement on non-crossing partitions.


\section{The canonical complex of a finite semidistributive lattice}
\label{sec:canonicalComplex}

This section deals with canonical meet and join representations in a finite semidistributive lattice and its quotients.
We start with a recollection on join semidistributive lattices, their canonical join representations, their canonical join complexes, their Kreweras maps, and their lattice congruences (\cref{subsec:recollectionLattices}).
We then define the canonical complex of a semidistributive lattice~$L$ which encodes the intervals of~$L$ and contains both the canonical join complex and the canonical meet complex of~$L$ (\cref{subsec:canonicalComplex}).


\subsection{Recollection on lattices}
\label{subsec:recollectionLattices}

We start by a quick recollection on semidistributive lattices, canonical representations, canonical complexes, Kreweras maps and lattice congruences.
All the material covered here is classical, we refer for instance to~\cite{FreeseNation, Reading-PosetRegionsChapter, Reading-arcDiagrams, Barnard}.
Following \mbox{\cite[Exm.~10]{Barnard}}, we illustrate this section with the case of distributive lattices.

\subsubsection{Join representations and semidistributive lattices}

Consider a finite lattice~$(L, \le, \join, \meet)$ where $\join$ is the join and $\meet$ is the meet.
We see $\join$ and $\meet$ as internal binary operators on~$L$ and try to factorize the elements of~$L$ in some canonical way.
It is first important to understand the irreducible elements for~$\join$ and~$\meet$.

\begin{definition}
An element~$x \in L$ is called \defn{join} (resp.~\defn{meet}) \defn{irreducible} if it covers (resp.~is covered by) a unique element denoted~$x_\star$ (resp.~$x^\star$).
We denote by $\JI$ (resp.~$\MI$) the subposet of~$L$ induced by the set of join (resp.~meet) irreducible elements of~$L$.
\end{definition}

\begin{definition}
A \defn{join representation} of~$x \in L$ is a subset~$J \subseteq L$ such that~$x = \bigJoin J$.
Such a representation is \defn{irredundant} if~$x \ne \bigJoin J'$ for any strict subset~$J' \subsetneq J$.
The irredundant join representations in~$L$ are antichains of~$L$, and are ordered by containement of the lower sets of their elements (\ie~$J \le J'$ if and only if for any~$y \in J$ there exists~$y' \in J'$ such that~$y \le y'$ in~$L$).
The \defn{canonical join representation} of~$x$, denoted~$\CJR(x)$, is the minimal irredundant join representation of~$x$ for this order, when it exists.
\end{definition}

Note that when it exists, $\CJR(x)$ is an antichain of~$\JI$.
The following statement characterizes the lattices where canonical join representations exist.

\begin{proposition}[{\cite[Thm.~2.24 \& Thm.~2.56]{FreeseNation}}]
\label{prop:semidistributive}
A finite lattice~$L$ is \defn{join semidistributive} when the following equivalent conditions hold:
\begin{enumerate}[(i)]
\item $x \join y = x \join z$ implies $x \join (y \meet z) = x \join y$ for any~$x, y, z \in L$,
\item for any cover relation~$x \lessdot y$ in~$L$, the set \[K_\join(x,y) \eqdef \set{z \in L}{z \not\le x \text{ but } z \le y} = \set{z \in L}{x \join z = y}\] has a unique minimal element~$k_\join(x,y)$ (which is then automatically join irreducible),
\item any element of~$L$ admits a canonical join representation.
\end{enumerate}
Moreover, the canonical join representation of~$y \in L$ is~$\CJR(y) = \set{k_\join(x, y)}{x \lessdot y}$.
\end{proposition}

Note that in a finite join semidistributive lattice~$L$, we can associate to any meet irreducible element~$m$ of~$L$ a join irreducible element~$\kappa_\join(m) \eqdef k_\join(m, m^\star)$ of~$L$.
Moreover, the existence of canonical join representations enable us to consider the following complex, illustrated in \cref{fig:CJCs-CMCs}.
It was initially defined in \cite{Reading-arcDiagrams}, where a combinatorial model was provided for the weak order (see \cref{subsec:NCADs}), and studied in~\cite{Barnard} for arbitrary finite semidistributive lattices.

\begin{definition}
The \defn{canonical join complex}~$\CJC$ of a finite join semidistributive lattice~$L$ is the simplicial complex on~$\JI$ whose faces are the canonical join representations of the elements~of~$L$.
\end{definition}

\begin{figure}
	\centerline{\includegraphics{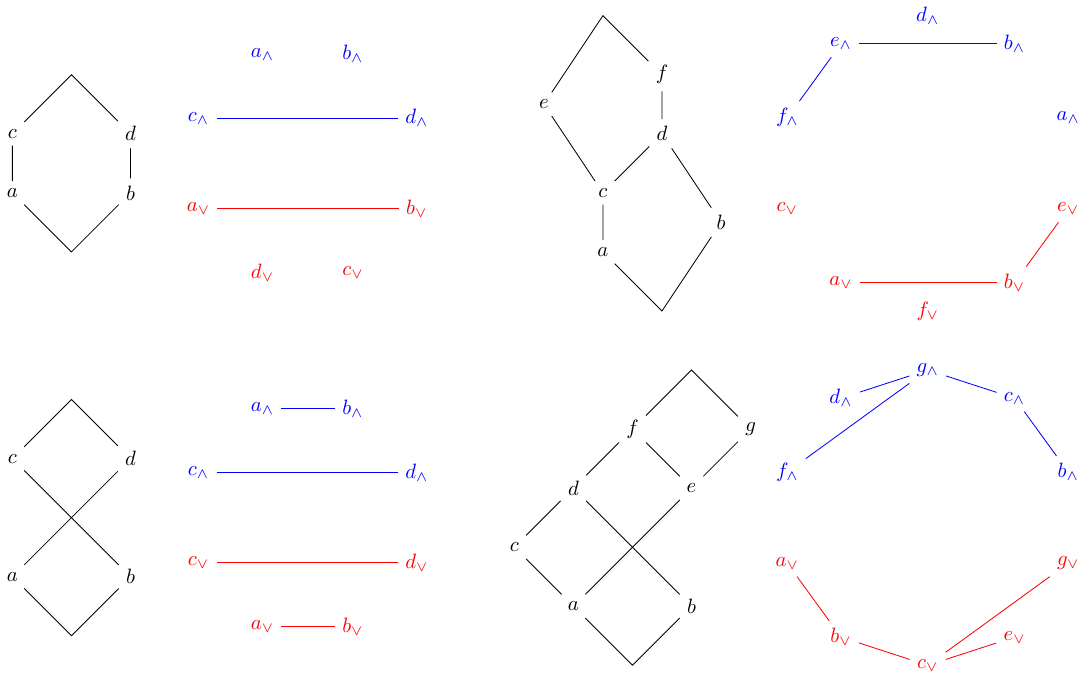}}
	\caption{Some semidistributive lattices and their canonical join (red) and meet (blue) complexes. The letters label all join or meet irreducible elements, and we denote by~$x_\join$ (resp.~$x_\meet$) the element~$x$ when it is considered as a join (resp.~meet) irreducible. Note that we always consistently color joinands in red and meetands in blue. The bottom two lattices are distributive while the top two are only semidistributive.}
	\label{fig:CJCs-CMCs}
\end{figure}

The \defn{meet semidistributivity}, the maps~$K_\meet$, $k_\meet$ and~$\kappa_\meet$, the \defn{canonical meet representation}~$\CMR(x)$ and the \defn{canonical meet complex}~$\CMC$ are all defined dually.
A lattice~$L$ is \defn{semidistributive} if it is both meet and join semidistributive.
In this case, the maps~$\kappa_\join$ and~$\kappa_\meet$ define inverse bijections between~$\MI$ and~$\JI$, and the complexes $\CJC$ and~$\CMC$ behave particularly nicely.

\pagebreak

\begin{proposition}[{\cite[Thm.~2 \& Coro.~5]{Barnard}}]
\label{prop:canonicalJoinComplex}
If~$L$ is a finite semidistributive lattice, then
\begin{enumerate}[(i)]
\item $\CJC$ and~$\CMC$ are flag simplicial complexes (\ie their minimal non-faces are edges, or equivalently they are the clique complexes of their graphs),
\item the maps~$\kappa_\join$ and~$\kappa_\meet$ induce inverse isomorphisms between $\CMC$ and~$\CJC$.
\end{enumerate}
\end{proposition}

In fact, it was proved in~\cite[Thm.~2]{Barnard} that~$\CJC$ is flag if and only if $L$ is semidistributive.
We will not use the ``only if'' direction in this paper.

\begin{example}[Distributive lattices]
\label{exm:kappaDistributiveLattice}
The name semidistributivity actually comes from the well understood class of distributive lattices.
A lattice~$L$ is \defn{distributive} if $x \join (y \meet z) = (x \join y) \meet (x \join z)$ for any~$x, y, z \in L$.
Note that the dual condition $x \meet (y \join z) = (x \meet y) \join (x \meet z)$ for any~$x, y, z \in L$ is actually equivalent to the primal one.
The fundamental theorem for distributive lattices affirms that $L$ is distributive if and only if it is isomorphic to the lattice of lower sets of its join irreducible poset~$P$.
In other words, any antichain of join irreducible elements in~$P$ forms a canonical join representation in~$L$.
To be more precise, consider, for an antichain~$A$ of~$P$, the two lower sets
\[
j_A \eqdef \set{x \in P}{x \le y \text{ for some } y \in A}
\qquad\text{and}\qquad
m^A \eqdef \set{x \in P}{x \not\ge y \text{ for all } y \in A}.
\]
Said differently, $A$ is the set of maximal elements of~$j_A$ and the set of minimal elements of~$P \ssm m^A$.
For~$y \in P$, we abbreviate~$j_{\{y\}}$ into~$j_y$ and~$m^{\{y\}}$ into~$m^y$.
Then 
\begin{itemize}
\item the join (resp.~meet) irreducibles of~$L$ are precisely the lower sets~$j_y$ (resp.~$m^y$) for~$y \in P$, 
\item the map~$\kappa_\join$ (resp.~$\kappa_\meet$) is given by $\kappa_\join(m^y) = j_y$ (resp.~$\kappa_\meet(j_y) = m^y$),
\item the canonical join representation of~$j_A$ and the canonical meet representation of~$m^A$ are
\[
\CJR(j_A) = \bigset{j_y}{y \in A}
\qquad\text{and}\qquad
\CMR(m^A) = \bigset{m^y}{y \in A}.
\]
\item the canonical join and meet complexes~$\CJC$ and~$\CMC$ are both (isomorphic to) the clique complex on the incomparability graph of~$P$.
\end{itemize}
See \cite[Exm.~10]{Barnard}.
\end{example}

\subsubsection{Kreweras maps}
\label{subsubsec:Kreweras}

In a semidistributive lattice~$L$, each element has both a canonical join representation and a canonical meet representation.
It is natural to consider the maps that exchange the canonical join representation with the canonical meet representation of the same element.

\begin{definition}
The \defn{Kreweras maps}~$\eta_\join : \CMC \to \CJC$ and~$\eta_\meet : \CJC \to \CMC$ are defined by
\[
\eta_\join(M) \eqdef \CJR \big( \bigMeet M \big)
\qquad\text{and}\qquad
\eta_\meet(J) \eqdef \CMR \big( \bigJoin J \big).
\]
\end{definition}

Note that some authors call Kreweras maps the compositions~$\eta_\join \circ \kappa_\meet : \CJC \to \CJC$ and $\eta_\meet \circ \kappa_\join : \CMC \to \CMC$, see for instance~\cite{Barnard}.
As will be discussed in \cref{exm:KrewerasComplement}, the Kreweras maps for the Tamari lattice are closely related to the classical Kreweras complement on non-crossing partitions.
For the moment, we recall that the Kreweras maps for the distributive lattices are related to rowmotion.

\begin{example}[Distributive lattices]
\label{exm:KrewerasDistributiveLattice}
With the notations of~\cref{exm:kappaDistributiveLattice}, for an antichain~$A$ in~$P$, we denote by $\row_\join(A)$ the set of maximal elements of~$m^A$ and by $\row_\meet(A)$ the set of minimal elements of~$P \ssm j_A$.
In other words, we have~$m^A = j_{\row_\join(A)}$ and~$j_A = m_{\row_\meet(A)}$.
Hence, by \cref{exm:kappaDistributiveLattice}, the Kreweras maps~$\eta_\join$ and~$\eta_\meet$ are given by
\[
\eta_\join(\set{m^y}{y \in A}) = \set{j_y}{y \in \row_\join(A)}
\qquad\text{and}\qquad
\eta_\meet(\set{j_y}{y \in A}) = \set{m^y}{y \in \row_\meet(A)}.
\]
See \cite[Rem.~32]{Barnard}.
\end{example}

\subsubsection{Lattice congruences}

We now discuss quotients of the lattice $L$, considered as an algebraic structure with two internal binary operators~$\join$ and~$\meet$.
We thus need equivalence relations on~$L$ that respects~$\join$ and~$\meet$.

\begin{definition}
A \defn{congruence}~$\equiv$ on~$L$ is an equivalence relation on~$L$ such that $x \equiv x'$ and $y \equiv y'$ implies $x \join y \equiv x' \join y'$ and $x \meet y \equiv x' \meet y'$.
Equivalently, the equivalence classes are intervals, and the maps~$\projDown[\equiv]$ and~$\projUp[\equiv]$ sending an element to the minimum and maximum elements in its congruence class are order preserving.
\end{definition}

\begin{definition}
The \defn{lattice quotient}~$L/{\equiv}$ is the lattice structure on the congruence classes, where for any two congruence classes~$X$ and~$Y$, 
\begin{itemize}
\item the order is given by~$X \le Y$ if and only if $x \le y$ for some representatives~$x \in X$ and~$y \in Y$,
\item the join~$X \join Y$ (resp.~meet~$X \meet Y$) is the congruence class of~$x \join y$ (resp.~$x \meet y$) for any representatives~$x \in X$ and~$y \in Y$.
\end{itemize}
\end{definition}

Note that the lattice quotient~$L/{\equiv}$ is isomorphic to the subposet of~$L$ induced by the minimal (or maximal) elements in their congruence classes.
This subposet is a join (resp.~meet) subsemilattice of~$L$ but may fail to be a sublattice of~$L$.
We now consider all congruences of~$L$.

\begin{definition}
The \defn{congruence lattice}~$\con(L)$ is the set of all congruences of~$L$ ordered by refinement.
\end{definition}

The congruence lattice~$\con(L)$ is a distributive lattice where the meet is the intersection of relations and the join is the transitive closure of union of relations.
For any join irreducible element~$j \in \JI$, we denote by~$\con(j)$ the unique minimal congruence of~$L$ that \defn{contracts}~$j$, that is with~$j_\star \equiv j$.
It turns out that~$\con(j)$ is join irreducible in~$\con(L)$ and that all join irreducible congruences in~$\con(L)$ are of this form.
Hence, any congruence of~$L$ is completely determined by the set of join irreducible elements of~$L$ that it contracts.
We denote by~$\UJI$ the set of join irreducible elements of~$L$ uncontracted by~$\equiv$.
Not all subsets of join irreducible elements of~$L$ are of the form~$\UJI$ for some congruence~$\equiv$ of~$L$.
The possible subsets are governed by the following relation.

\begin{definition}
For~$j,j' \in \JI$, we say that~$j$ \defn{forces}~$j'$, and write~$j \succcurlyeq j'$, if $\con(j) \ge \con(j')$, that is if any congruence contracting~$j$ also contracts~$j'$.
\end{definition}

The forcing relation is a preorder~$\preccurlyeq$ (\ie a transitive and reflexive, but not necessarily antisymmetric, relation) on $\JI$, whose upper sets correspond to the congruences of~$L$.

\begin{proposition}[{\cite[Prop.~9-5.16]{Reading-PosetRegionsChapter}}]
The following conditions are equivalent for~$J \subseteq \JI$:
\begin{itemize}
\item $J$ is an upper set of the forcing preorder (\ie $j \succcurlyeq j'$ and~$j \in J$ implies~$j' \in J$).
\item $J = \UJI$ for some congruence~$\equiv$ of~$L$.
\end{itemize}
\end{proposition}

As already mentioned, the set~$\UJI$ characterizes~$\equiv$. 
It moreover enables to understand the elements of~$L$ which are minimal in their congruence classes and their canonical join and meet representations as follows.

\begin{proposition}[{\cite[Prop.~9-5.29]{Reading-PosetRegionsChapter}}]
\label{prop:canonicalJoinRepresentationsQuotient}
Let~$\equiv$ be a congruence of a finite join semidistributive lattice~$L$.
Then
\begin{itemize}
\item an element~$x \in L$ is minimal in its congruence class if and only if $\CJR(x) \subseteq \UJI$,
\item the quotient~$L/{\equiv}$ is join semidistributive and the canonical joinands of a congruence~class~$X$ in~$L/{\equiv}$ are the congruence classes of the canonical joinands of the minimal element~in~$X$.
\end{itemize}
\end{proposition}

\cref{prop:canonicalJoinRepresentationsQuotient} translates as follows to the canonical join complex~$\CJC$.

\begin{proposition}
\label{prop:canonicalJoinComplexQuotient}
Let~$\equiv$ be a congruence on a finite join semidistributive lattice~$L$.
Then the canonical join complex~$\CJC[L/{\equiv}]$ of the quotient~$L/{\equiv}$ is isomorphic to the subcomplex~$\CJC[\equiv]$ of the canonical join complex~$\CJC$ of~$L$ induced by~$\UJI$.
\end{proposition}

We will need the following statement relating the canonical join representation of an element~$x$ of~$L$ with the canonical join representation of the minimal element~$\projDown[\equiv](x)$ in its equivalence class.
We provide a proof here as we have not found this statement explicitly in the literature.

\begin{proposition}
\label{prop:canonicalJoinRepresentationProjection}
Let~$\equiv$ be a congruence of a finite join semidistributive lattice~$L$.
For any~$x \in L$, the lower ideal of~$L$ generated by~$\CJR(x)$ contains~$\CJR \big( \projDown[\equiv](x) \big)$.
\end{proposition}

\begin{proof}
The proof works by induction on the size of the interval~$[\projDown[\equiv](x), x]$. 
The statement is immediate if~$\projDown[\equiv](x) = x$.
Otherwise, there exists~$j \in \CJR(x) \ssm \UJI$.
Let~$y \eqdef \bigJoin \big( \CJR(x) \symdif \{j, j_\star\} \big)$, where~$\symdif$ denotes the symmetric difference.
Since~$j \equiv j_\star$, we have~$x \equiv y$ and thus~$\projDown[\equiv](x) = \projDown[\equiv](y)$.
Since~$y$ has a join representation strictly contained in the lower ideal of~$L$ generated by~$\CJR(x)$, we have~$y < x$ so that~$[\projDown[\equiv](y), y]$ is strictly contained in~$[\projDown[\equiv](x), x]$.
By induction hypothesis, we thus obtain that~$\CJR \big( \projDown[\equiv](y) \big)$ is contained in the lower ideal of~$L$ generated by~$\CJR(y)$.
Moreover, observe that~$J \eqdef \big( \CJR(x) \ssm \{j\} \big) \cup \CJR(j_\star)$ is a join representation for~$y$, so that~$\CJR(y)$ is contained in the lower ideal of~$L$ generated by~$J$, which is itself contained in the lower ideal of~$L$ generated by~$\CJR(x)$ (since any element of~$\CJR(j_\star)$ is lower than~$j_\star$ and thus than~$j$).
We conclude that~$\CJR \big( \projDown[\equiv](x) \big)$ is indeed in the lower ideal of~$L$ generated by~$\CJR(x)$.
\end{proof}

Note that this property is quite specific to~$\projDown[\equiv](x)$.
Namely, a relation $x \le y$ does not imply any inclusion between the lower ideals of~$L$ generated by~$\CJR(x)$ and~$\CJR(y)$ in general (see \eg \cref{fig:CJCs-CMCs}).
\cref{prop:canonicalJoinRepresentationProjection} ensures that we can look for~$\CJR \big( \projDown[\equiv](x) \big)$ among the antichains of join irreducible elements of~$L$ uncontracted by~$\equiv$ and below a join irreducible element of~$\CJR(x)$.
Unfortunately, it remains difficult in general to describe~$\CJR \big( \projDown[\equiv](x) \big)$ because not all such antichains define a canonical join representation~of~$L$.

Dual statements hold using meets instead of joins, and we denote by~$\UMI$ the meet irreducible elements of~$L$ uncontracted by~$\equiv$, and by~$\CMC[\equiv]$ the subcomplex of~$\CMC$ induced by~$\UMI$ for a congruence~$\equiv$ on~$L$.
Due to \cref{prop:canonicalJoinComplexQuotient}, we will always work with the subcomplexes~$\CJC[\equiv]$ and~$\CMC[\equiv]$ rather than with the complexes~$\CJC[L/{\equiv}]$ and~$\CMC[L/{\equiv}]$.

When the lattice~$L$ is semidistributive, the two sets~$\UJI$ and~$\UMI$ and the two subcomplexes~$\CJC[\equiv]$ and~$\CMC[\equiv]$ are connected by the maps~$\kappa_\join$ and~$\kappa_\meet$.
We provide a proof here as we have not found this statement explicitly in the literature.

\begin{proposition}
\label{prop:uncontractedElements}
Let~$\equiv$ be a congruence on a finite semidistributive lattice~$L$.
Then we have ${\UJI = \kappa_\join \big( \UMI \big)}$ and~$\UMI = \kappa_\meet \big( \UJI \big)$.
Hence, the maps~$\kappa_\join$ and~$\kappa_\meet$ induce inverse isomorphisms between the subcomplexes~$\CMC[\equiv]$ and~$\CJC[\equiv]$.
\end{proposition}

\begin{proof}
Let~$m \in \MI$ and~$j = \kappa_\join(m)$.
By definition, we have~$m \join j = m^\star$ and~$m \join j_\star = m$.
Hence, $j \equiv j_\star$ implies that~$m = m \join j_\star \equiv m \join j = m^\star$.
In other words, $\kappa_\join \big( \UMI \big) \subseteq \UJI$.
By symmetry, we have~$\kappa_\meet \big( \UJI \big) \subseteq \UMI$.
Since~$\kappa_\join$ and~$\kappa_\meet$ are reversed bijections, this yields equalities.
The last sentence of the statement thus follows from \cref{prop:canonicalJoinComplex}\,(ii).
\end{proof}

\begin{example}[Distributive lattices]
\label{exm:congruencesDistributiveLattice}
In a distributive lattice~$L$, there is no forcing at all.
Hence, any subset of join irreducible elements of~$L$ defines a congruence of~$L$.
In other words, with the notations of~\cref{exm:kappaDistributiveLattice}, any subset~$Y$ of~$P$ defines a congruence~$\equiv_Y$ with
\[
\UJI[\equiv_Y] = \set{j_y}{y \in Y}
\quad\text{and}\quad
\UMI[\equiv_Y] = \set{m^y}{y \in Y}.
\]
The lattice quotient~$L/{\equiv}$ is again distributive and isomorphic to the lattice of lower ideals of the restriction of the poset~$P$ to~$Y$.
\end{example}


\subsection{The canonical complex}
\label{subsec:canonicalComplex}

We now define another complex that connects the canonical join complex~$\CJC$ to the canonical meet complex~$\CMC$ using intervals of~$L$.
This complex is illustrated in \cref{fig:CCs}.

\begin{definition}
\label{def:canonicalComplex}
The \defn{canonical complex}~$\CC$ of a finite semidistributive lattice~$L$ is the simplicial complex whose
\begin{itemize}
\item ground set is the \textbf{disjoint} union~$\JI \sqcup \MI$ of the sets of join irreducible and of meet irreducible elements of~$L$, and
\item faces are the \textbf{disjoint} unions~$J \sqcup M$ where~$J \in \CJC$ is a canonical join representation, $M \in \CMC$ is a canonical meet representation, and~$\bigJoin J \le \bigMeet M$.
\end{itemize}
\end{definition}

\begin{figure}
	\centerline{\includegraphics{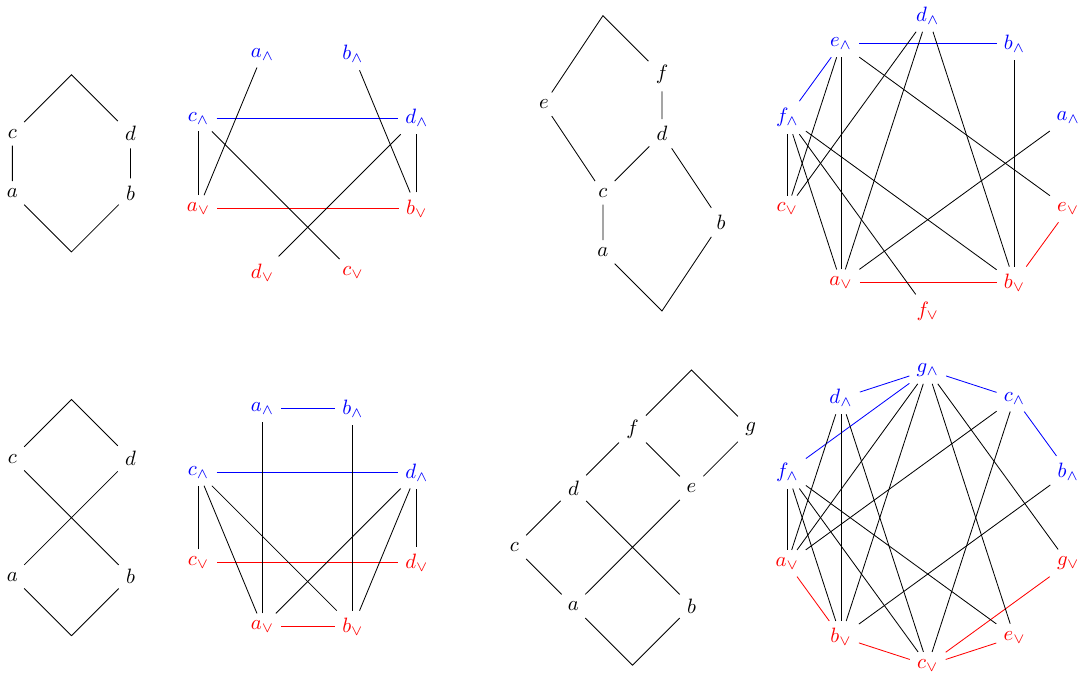}}
	\caption{The canonical complexes of the semidistributive lattices of \cref{fig:CJCs-CMCs}. The corresponding join (resp.~meet) canonical complexes of \cref{fig:CJCs-CMCs} are highlighted in red (resp.~blue). Since the canonical complexes are flag by \cref{prop:canonicalJoinComplex,prop:canonicalComplexFlag}, it is sufficient to represent their graphs. The letters label all join or meet irreducible elements, and we denote by~$x_\join$ (resp.~$x_\meet$) the element~$x$ when it is considered as a join (resp.~meet) irreducible. The vertices of the canonical complex are positioned so that the map $\kappa$ of \cref{rem:symmetryKappa} acts by central symmetry.}
	\label{fig:CCs}
\end{figure}

To avoid any confusion, let us insist here that when an element~$x \in L$ is both join irreducible and meet irreducible, then it appears twice in~$\CC$ and may appear twice in some faces of~$\CC$, once as a join irreducible element and once as a meet irreducible element.
We will thus always write the faces of~$\CC$ explicitly as disjoint unions.
In our pictures, we denote by~$x_\join$ (resp.~$x_\meet$) and color red (resp.~blue) the element~$x \in L$ considered as a join (resp.~meet) irreducible.

Our first observation is that, while the canonical join and meet complexes~$\CJC$ and~$\CMC$ encode the individual elements of~$L$, the canonical complex~$\CC$ encodes the intervals of~$L$.

\begin{definition}
The \defn{canonical representation} of an interval~$[x,y]$ of a semidistributive lattice~$L$ is the \textbf{disjoint} union~$\CJR(x) \sqcup \CMR(y)$.
\end{definition}

\begin{proposition}
\label{prop:intervals}
For a finite semidistributive lattice~$L$, the faces of the canonical complex~$\CC$ are precisely the canonical representations of the intervals of~$L$.
\end{proposition}

\begin{proof}
An interval~$[x,y]$ of~$L$ corresponds to a face~$\CJR(x) \sqcup \CMR(y)$ of~$\CC$ since~$\CJR(x) \in \CJC$, $\CMR(y) \in \CMC$ and~$\bigJoin \CJR(x) = x \le y =  \bigMeet \CMR(y)$.
Conversely, a face~$J \sqcup M$ of~$\CC$ corresponds to the interval~$[\bigJoin J, \bigMeet M]$ of~$L$.
\end{proof}

We next observe that~$\CC$ contains both~$\CJC$ and~$\CMC$ as induced subcomplexes.

\begin{proposition}
\label{prop:canonicalJoinMeetComplexesSubcomplexes}
For a finite semidistributive lattice~$L$, the canonical join (resp.~meet) complex $\CJC$ (resp.~$\CMC$) is the subcomplex of the canonical complex~$\CC$ induced by the join (resp.~meet) irreducible elements~$\JI$ (resp.~$\MI$).
\end{proposition}

\begin{proof}
We have $J \in \CJC \iff J \sqcup \varnothing \in \CC$ and $M \in \CMC \iff \varnothing \sqcup M \in \CC$.
\end{proof}

Our next statement is the analogue of \cref{prop:canonicalJoinComplex}\,(i).

\begin{proposition}
\label{prop:canonicalComplexFlag}
For a finite semidistributive lattice~$L$, the canonical complex~$\CC$ is a flag simplicial complex.
\end{proposition}

\begin{proof}
Consider~$J \subseteq J'$ and~$M \subseteq M'$ such that~$J' \sqcup M' \in \CC$.
Then~$J \in \CJC$ since $J' \in \CJC$ and~$M \in \CMC$ since~$M' \in \CMC$ (since~$\CJC$ and~$\CMC$ are simplicial complexes), and~$\bigJoin J \le \bigJoin J' \le \bigMeet M' \le \bigMeet M$.
Hence $J \sqcup M \in \CC$ so that~$\CC$ is indeed a simplicial complex.

Consider~$J \subseteq \JI$ and~$M \subseteq \MI$ such that any two elements of~$J \sqcup M$ form a face of~$\CC$.
Then~$J \in \CJC$ and~$M \in \CMC$ (since~$\CJC$ and~$\CMC$ are flag by \cref{prop:canonicalJoinComplex}\,(i)), and~$\bigJoin J \le \bigMeet M$ since~$j \le m$ for any~$j \in J$ and~$m \in M$.
\end{proof}

Our next statement says that~$\CC$ is not only a simplicial complex, it is actually naturally embedded on the boundary of a cross-polytope.

\begin{proposition}
\label{prop:symmetryKappa}
For any~$j \in \JI$, the pair~$\{j, \kappa_\meet(j)\}$ is not in~$\CC$.
Hence, for any labeling ${\lambda : \JI \to [|\JI|]}$, the map sending~$j$ to~$\b{e}_{\lambda(j)}$ and $\kappa_\meet(j)$ to~$-\b{e}_{\lambda(j)}$ defines an embedding of~$\CC$ to the boundary of the $|\JI|$-dimensional cross-polytope.
\end{proposition}

\begin{proof}
By definition, $j \not\le \kappa_\meet(j)$ so that~$\{j, \kappa_\meet(j)\}$ is not in~$\CC$.
The second sentence thus follows from \cref{prop:canonicalComplexFlag}.
\end{proof}

\begin{remark}
\label{rem:symmetryKappa}
Denote by~$\kappa$ the map on~$\JI \sqcup \MI$ defined by~$\kappa(m) \eqdef \kappa_\join(m)$ for~$m \in \MI$ and~$\kappa(j) \eqdef \kappa_\meet(j)$ for~$j \in \JI$.
It corresponds to the central symmetry on the corresponding cross-polytope.
By \cref{prop:canonicalJoinComplex}\,(ii), the two subcomplexes $\CJC$ and $\CMC$ are symmetric under the action of~$\kappa$.
However, the full canonical complex~$\CC$ is not invariant under the action of~$\kappa$.
See \cref{fig:CCs} for examples.
\end{remark}

\begin{example}
The canonical complex of the boolean lattice on~$[n]$ is isomorphic to the boundary of the $n$-dimensional cross-polytope.
\end{example}

Finally, analogously to \cref{prop:canonicalJoinComplexQuotient}, the canonical complex is compatible with lattice congruences of~$L$.

\begin{proposition}
\label{prop:canonicalComplexQuotient}
For any congruence~$\equiv$ of a finite semidistributive lattice~$L$, the canonical complex~$\CC[L/{\equiv}]$ of the quotient~$L/{\equiv}$ is isomorphic to the subcomplex~$\CC[\equiv]$ of the canonical complex~$\CC$ of~$L$ induced by the disjoint union~$\UJI \sqcup \UMI$ of the join and meet irreducible elements of~$L$ uncontracted by~$\equiv$.
\end{proposition}

\begin{proof}
This immediately follows from \cref{prop:canonicalJoinComplexQuotient} and its dual version.
\end{proof}

\begin{example}[Distributive lattices]
\label{exm:canonicalComplexDistributiveLattice}
With the notations of~\cref{exm:kappaDistributiveLattice}, we have~$j_y \subseteq m^z \iff y \not\ge z$.
Hence, the canonical complex~$\CC$ is the clique complex of the graph whose vertex set is made of two copies~$P_\join$ and~$P_\meet$ of~$P$ and whose edge set is the union of two copies~$I_\join$ and~$I_\meet$ of the incomparability graph of~$P$ with the edges~$\{y_\join, z_\meet\}$ for~$y \not\ge z$ in~$P$.
\end{example}


\section{Semi-crossing arc bidiagrams}
\label{sec:bidiagrams}

In this section, we apply the lattice theoretic results presented in \cref{sec:canonicalComplex} to the classical weak order on permutations.
We first recall that the canonical join and meet representations of the weak order can be encoded as non-crossing arc diagrams as defined by N.~Reading in~\cite{Reading-arcDiagrams} (\cref{subsec:NCADs}).
We then briefly study the restriction of the weak order on join or meet irreducibles in terms of arcs (\cref{subsec:weakOrderArcs}).
We then describe the intervals and the canonical complex of the weak order in terms of semi-crossing arc diagrams (\cref{subsec:SCABs}).
We then provide direct bijections between the semi-crossing arc bidiagrams and the weak order interval posets of~\cite{ChatelPilaudPons} (\cref{subsec:WOIPs}).
Finally, we provide an algorithm to compute the Kreweras maps in any lattice quotient of the weak order, generalizing the classical Kreweras complement on non-crossing partitions (\cref{subsec:KrewerasBidiagrams}).


\subsection{Non-crossing arc diagrams}
\label{subsec:NCADs}

We consider the following classical order on the set~$\f{S}_n$ of permutations of $[n] \eqdef \{1, \dots, n\}$, illustrated in \cref{fig:weakOrder}.

\begin{definition}
An \defn{inversion} of a permutation $\sigma \in \f{S}_n$ is a pair $(u,v)$ with~$1 \le u < v \le n$ and~$\sigma^{-1}(u) > \sigma^{-1}(v)$ (in other words, $u$ is smaller than~$v$ but $u$ appears after~$v$ in~$\sigma$).
The (right) \defn{weak order} of size $n$ is the order on permutations of~$\f{S}_n$ defined by inclusion of their inversion sets.
\end{definition}

\begin{figure}
	\centerline{\includegraphics{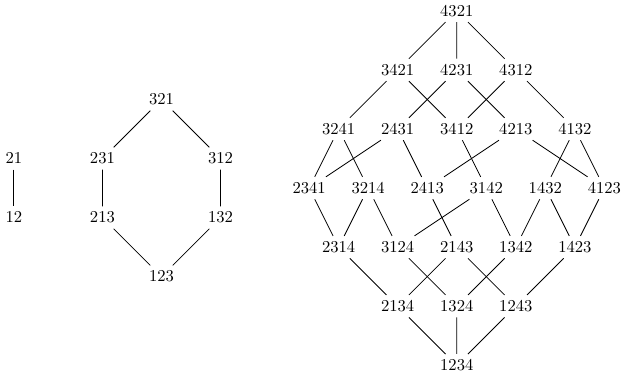}}
	\caption{Hasse diagrams of the right weak orders of size $2$, $3$, and $4$.}
	\label{fig:weakOrder}
\end{figure}

Note that a cover relation in the weak order corresponds to the swap of two values~$\sigma_i$ and~$\sigma_{i+1}$ at consecutive positions.
The swap is increasing in the weak order if $i$ is an \defn{ascent} \ie $\sigma_i < \sigma_{i+1}$, and decreasing if $i$ is a \defn{descent} \ie $\sigma_i > \sigma_{i+1}$.

It is classical that the weak order is a semidistributive lattice (the lattice property was proved in~\cite{GuilbaudRosenstiehl,Bjorner}, the semidistributivity in~\cite{ContePolyBarbut}).
We now describe its join (resp.~meet) irreducible elements and its canonical join (resp.~meet) representations in terms of the arcs and non-crossing arc diagrams introduced by N.~Reading in~\cite{Reading-arcDiagrams}.

\begin{definition}[\cite{Reading-arcDiagrams}]
\label{def:NCAD}
An \defn{arc} is a quadruple~$(a,b,A,B)$ where $1 \le a < b \le n$ and ${A \sqcup B = {]a,b[}}$ forms a partition of~$]a,b[ \eqdef \{a+1, \dots, b-1\}$.
Two arcs~$\alpha \eqdef (a, b, A, B)$ and~$\alpha' \eqdef (a', b', A', B')$ \defn{cross} if there exist~$u \ne v$ such that ${u \in (A' \cup \{a', b'\}) \cap (B \cup \{a, b\})}$ and ${v \in (A \cup \{a, b\}) \cap (B' \cup \{a', b'\})}$.
A \defn{non-crossing arc diagram} (or \defn{NCAD} for short) is a collection of pairwise non-crossing arcs.
The \defn{non-crossing complex} is the clique complex of the non-crossing relation on all arcs.
\end{definition}

\begin{remark}
Visually, an arc~$(a, b, A, B)$ is represented by an $x$-monotone curve wiggling around the horizontal axis, starting at~$a$ and ending at~$b$, and passing above points of $A$ and below points of $B$. 
Two arcs cross if they cross in their interiors or start at the same point or end at the same point (but they do not cross if one ends where the other starts).
See \cref{fig:diagExamples} for illustrations of arcs and of non-crossing arc diagrams.
\end{remark}

Observe that the join (resp.~meet) irreducible elements of the weak order are precisely the permutations with exactly one descent (resp.~ascent).
Hence, we associate to an arc~$\alpha \eqdef (a, b, A, B)$ with~$A \eqdef \{a_1 < \dots < a_k\}$ and~$B \eqdef \{b_1 < \dots < b_\ell\}$
\begin{itemize}
\item a join irreducible permutation $\b{\sigma}_\join(\alpha) \eqdef [1, \dots, (a-1), a_1, \dots, a_k, b, a, b_1, \dots, b_\ell, (b+1), \dots, n]$, 
\item a meet irreducible permutation $\b{\sigma}_\meet(\alpha) \eqdef [n, \dots, (b+1), a_k, \dots, a_1, a, b, b_\ell, \dots, b_1, (a-1), \dots, 1]$,
\end{itemize}
where we use the one-line notation of permutations~$\sigma = [\sigma_1, \dots, \sigma_n]$.
\begin{figure}[h]
	\centerline{\includegraphics[scale=.85]{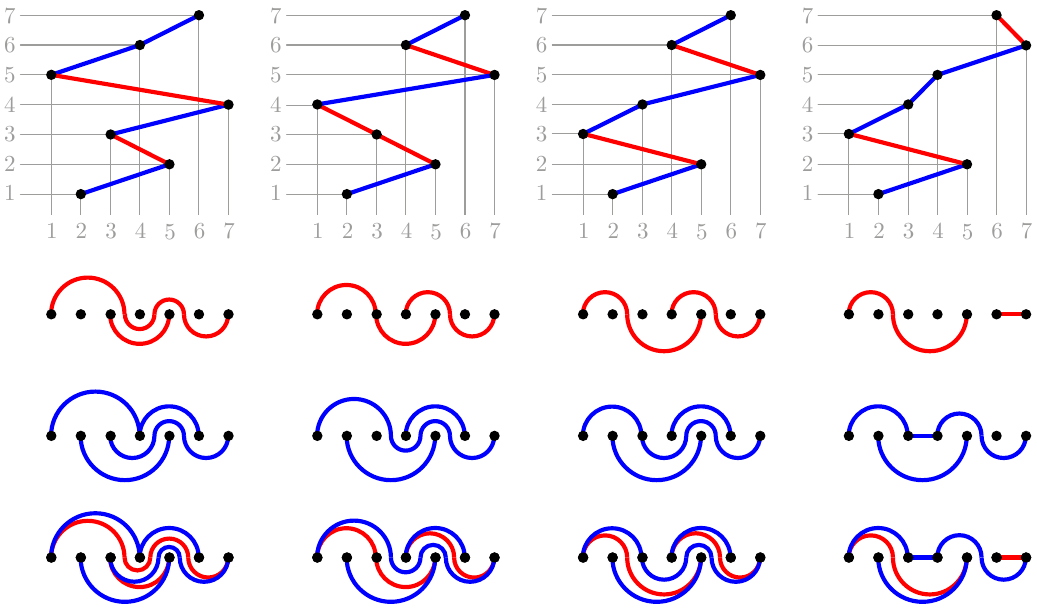}}
	\caption{NCADs and SCABs of the permutations $2537146$, $2531746$, $2513746$, and $2513476$. The first line represents the table~$(\sigma(i), i)$ of a permutation~$\sigma$ with ascents in blue and descents in red, the second line is the join diagram $\b{\delta}_\join(\sigma)$, the third line is the meet diagram $\b{\delta}_\meet(\sigma)$, and the fourth line is the superimposition~$\b{\delta}_\join(\sigma) \sqcup \b{\delta}_\meet(\sigma)$.}
	\label{fig:diagExamples}
\end{figure}

Consider now a permutation~$\sigma \in \f{S}_n$ represented by its permutation table formed by dots at coordinates~$(\sigma_i,i)$ for~$i \in [n]$.
Draw segments between consecutive dots~$(\sigma_i, i)$ and~$(\sigma_{i+1}, i+1)$, colored red for a descent~$\sigma_i > \sigma_{i+1}$ and blue for an ascent~$\sigma_i < \sigma_{i+1}$.
Finally, flatten the picture vertically to the horizontal line, allowing segments to bend but not to pass points.
The resulting picture is the superimposition of a set~$\b{\delta}_\join(\sigma)$ of red arcs and a set~$\b{\delta}_\meet(\sigma)$ of blue arcs.
See \cref{fig:diagExamples}.
More formally, $\b{\delta}_\join(\sigma) \eqdef \set{\b{\alpha}_\join(\sigma, i)}{\sigma_i > \sigma_{i+1}}$ and $\b{\delta}_\meet(\sigma) \eqdef \set{\b{\alpha}_\meet(\sigma, i)}{\sigma_i < \sigma_{i+1}}$ 
where
\begin{align*}
\b{\alpha}_\join(\sigma, i) & \eqdef (\sigma_{i+1}, \sigma_i, \set{\sigma_j}{j < i \text{ and } \sigma_i > \sigma_j > \sigma_{i+1}}, \set{\sigma_j}{j > i+1 \text{ and } \sigma_i > \sigma_j > \sigma_{i+1} }), \\ 
\text{and}\quad
\b{\alpha}_\meet(\sigma, i) & \eqdef (\sigma_i, \sigma_{i+1}, \set{\sigma_j}{j < i \text{ and } \sigma_i < \sigma_j < \sigma_{i+1}}, \set{\sigma_j}{j > i+1 \text{ and } \sigma_i < \sigma_j < \sigma_{i+1}}).
\end{align*}

\begin{proposition}[\cite{Reading-arcDiagrams}]
\label{prop:NCAD}
The map~$\b{\delta}_\join$ (resp.~$\b{\delta}_\meet$) is a bijection between the set of permutations of~$\f{S}_n$ and the set of non-crossing arc diagrams of size~$n$.
Moreover, the canonical join (resp.~meet) representation of a permutation~$\sigma \in \f{S}_n$ is given by~$\CJR(\sigma) = \set{\b{\sigma}_\join(\alpha_\join)}{\alpha_\join \in \b{\delta}_\join(\sigma)}$ (resp.~$\CMR(\sigma) = \set{\b{\sigma}_\meet(\alpha_\meet)}{\alpha_\meet \in \b{\delta}_\meet(\sigma)}$).
Hence, the canonical join (resp.~meet) complex of the weak order is isomorphic to the non-crossing complex.
\end{proposition}

By construction, the non-crossing arc diagrams are adapted to the maps~$\kappa_\join$ and~$\kappa_\meet$ and to quotients of the weak order.
First, our next statement says that the maps~$\kappa_\join$ and~$\kappa_\join$ only change the colors of the arcs.
We provide a proof as we have not found it explicitly in~\cite{Reading-arcDiagrams}.

\begin{proposition}
\label{prop:kappaChangesColor}
$\kappa_\join(\b{\sigma}_\meet(\alpha)) = \b{\sigma}_\join(\alpha)$ and~$\kappa_\meet(\b{\sigma}_\join(\alpha)) = \b{\sigma}_\meet(\alpha)$ for any arc~$\alpha$.
\end{proposition}

\begin{proof}
Let~$\alpha \eqdef (a, b, A, B)$, let~$\sigma \eqdef \b{\sigma}_\meet(\alpha) = [n, \dots, (b+1), a_k, \dots, a_1, a, b, b_\ell, \dots, b_1, (a-1), \dots, 1]$ and let~$\sigma^\star = [n, \dots, (b+1), a_k, \dots, a_1, b, a, b_\ell, \dots, b_1, (a-1), \dots, 1]$ denote the only element covering~$\sigma$.
Any permutation~$\tau$ with~$\tau \not\le \sigma$ but~$\tau \le \sigma^\star$ must have all elements of~$A$ before~$b$ before~$a$ before all elements of~$B$.
Therefore, the minimal such permutation is clearly $\b{\sigma}_\join(\alpha) = [1, \dots, (a-1), a_1, \dots, a_k, b, a, b_1, \dots, b_\ell, (b+1), \dots, n]$.
\end{proof}

\begin{proposition}[\cite{Reading-arcDiagrams}]
\label{prop:subarcs}
For any arcs~$\alpha \eqdef (a, b, A, B)$ and~$\alpha' \eqdef (a', b', A', B')$, the join irreducible~$\b{\sigma}_\join(\alpha)$ forces the join irreducible~$\b{\sigma}_\join(\alpha')$ if and only if~$\alpha$ is a \defn{subarc} of~$\alpha'$, meaning that~$a' \le a < b \le b'$ and~$A \subseteq A'$ while~$B \subseteq B'$.
Hence, to each upper ideal~$I$ of the subarc order corresponds a lattice congruence~$\equiv_I$ of the weak order, and the canonical join (resp.~meet) complex of the quotient of the weak order by~$\equiv_I$ is isomorphic to the non-crossing complex on~$I$.
\end{proposition}

\begin{remark}
\begin{figure}
	\centerline{\includegraphics[scale=.5]{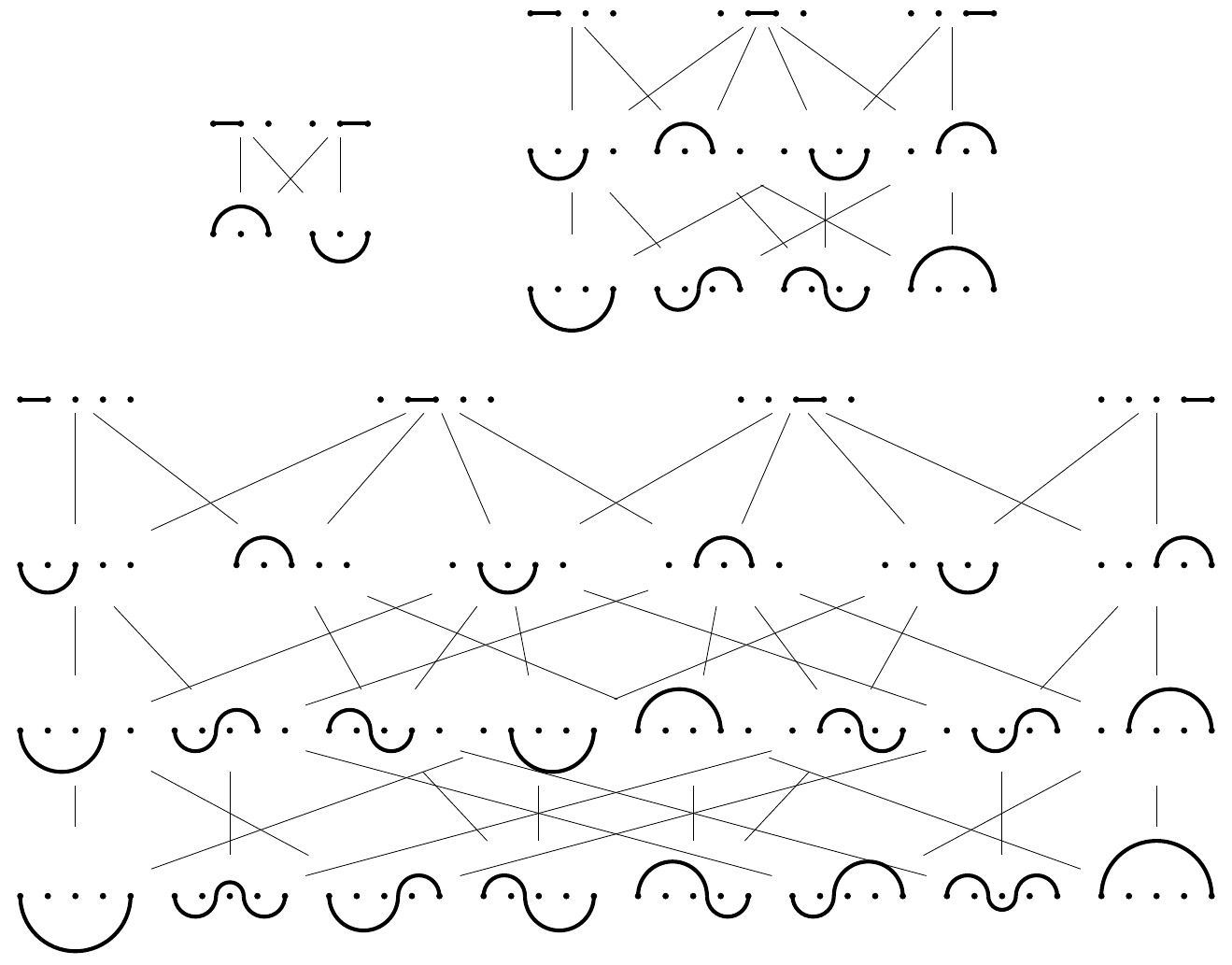}}
	\caption{The subarc order on arcs of sizes $3$ (top left), $4$ (top right), and $5$ (bottom).}
	\label{fig:forcing35}
\end{figure}
Visually, $\alpha$ is a subarc of~$\alpha'$ if the endpoints of~$\alpha$ are weakly in between the endpoints of~$\alpha'$, and $\alpha$ follows $\alpha'$ between its endpoints.
The subarc order on arcs of size~$3$ to~$5$ is represented in \cref{fig:forcing35}.
\end{remark}

\begin{example}
The prototypical congruence of the weak order is the \defn{sylvester congruence}~$\equiv_\textrm{sylv}$~\cite{LodayRonco,HivertNovelliThibon-algebraBinarySearchTrees}, which can be defined equivalently as
\begin{itemize}
\item the fiber of the binary search tree insertion (inserting a permutation from right to left),
\item the congruence where each class is the set of linear extensions of a binary tree (labeled in inorder and oriented toward its root),
\item the transitive closure of the rewriting rule~$UacVbW \equiv UcaVbW$ for~$1 \le a < b < c \le n$,
\item the congruence corresponding to the upper ideal of the subarc order given by all up arcs~$(a, b, {]a,b[}, \varnothing)$ (or equivalently, generated by the long up arc~$(1, n, {]1,n[}, \varnothing)$.
\end{itemize}
Hence, a permutation is minimal (resp.~maximal) in its class if and only if it avoids the pattern~$312$~(resp.~$132$).
The quotient of the weak order by the sylvester congruence is (isomorphic to) the classical \defn{Tamari lattice}~\cite{Tamari,HuangTamari}, whose elements are the binary trees on~$n$ nodes and whose cover relations are rotations in binary trees.
The canonical join representations in the Tamari lattice correspond to non-crossing sets of up arcs, also known as \defn{non-crossing partitions}.
The sylvester congruence was extended in~\cite{Reading-CambrianLattices} to Cambrian congruences and in~\cite{PilaudPons-permutrees} to permutree congruences.
\end{example}


\subsection{Weak order on arcs}
\label{subsec:weakOrderArcs}

We now briefly compare join or meet irreducible elements in the weak order in terms of arcs.
For this, we first observe that the inversions of~$\b{\sigma}_\join(\alpha)$ and~$\b{\sigma}_\meet(\alpha)$ are easily read on the arc~$\alpha$.

\begin{lemma}
\label{lem:arcInversions}
For any arc $\alpha \eqdef (a, b, A, B)$ and any~$u < v$, the pair~$(u, v)$ is an inversion of~$\b{\sigma}_\join(\alpha)$ (resp.~of~$\b{\sigma}_\meet(\alpha)$) if and only if~$u \in B \cup \{a\}$ and $v \in A \cup \{b\}$ (resp.~if~$u \notin A \cup \{a\}$ or~$v \notin B \cup \{b\}$).
\end{lemma}

\begin{proof}
Immediate from the definition~$\b{\sigma}_\join(\alpha) \eqdef [1, \dots, (a-1), a_1, \dots, a_k, b, a, b_1, \dots, b_\ell, (b+1), \dots, n]$ and~$\b{\sigma}_\meet(\alpha) \eqdef [n, \dots, (b+1), a_k, \dots, a_1, a, b, b_\ell, \dots, b_1, (a-1), \dots, 1]$.
\end{proof}

\begin{corollary}
\label{coro:weakOrderArcs}
For any two arcs~$\alpha \eqdef (a, b, A, B)$ and~$\alpha' \eqdef (a', b', A', B')$, we have
\begin{enumerate}[(i)]
\item $\b{\sigma}_\join(\alpha) \le \b{\sigma}_\join(\alpha')$ if and only if $a \in B' \cup \{a'\}$ and~$b \in A' \cup \{b'\}$, and~$A \subseteq A'$ and~$B \subseteq B'$,
\item $\b{\sigma}_\meet(\alpha) \le \b{\sigma}_\meet(\alpha')$ if and only if $a' \in B \cup \{a\}$ and~$b' \in A \cup \{b\}$, and~$A' \subseteq A$ and~$B' \subseteq B$,
\item $\b{\sigma}_\join(\alpha) \le \b{\sigma}_\meet(\alpha')$ if and only if there is no~$u < v$ such that~${u \in (A' \cup \{a'\}) \cap (B \cup \{a\})}$ and ${v \in (A \cup \{b\}) \cap (B' \cup \{b'\})}$.
\end{enumerate}
\end{corollary}

\begin{remark}
\label{rem:weakOrderArcs}
\cref{fig:induction35} shows the weak order on arcs defined by~$\alpha \le \alpha'$ if $\b{\sigma}_\join(\alpha) \le \b{\sigma}_\join(\alpha')$.
Visually, $\alpha \le \alpha'$ if $\alpha$ is a subarc of~$\alpha'$ which starts weakly below~$\alpha'$ and ends weakly above~$\alpha'$.
Note that~$\alpha \eqdef (a, b, A, B)$ covers at most two arcs, namely~${(\min B, b, A \cap {]\min b, b[}, B \ssm \min B)}$ and ${(a, \max A, A \ssm \max A, B \cap {]a, \max A[})}$ when they are defined.
Similar remarks hold for the order defined by~$\b{\sigma}_\meet(\alpha)$ instead of~$\b{\sigma}_\join(\alpha)$.
\begin{figure}
	\centerline{\includegraphics[scale=.25]{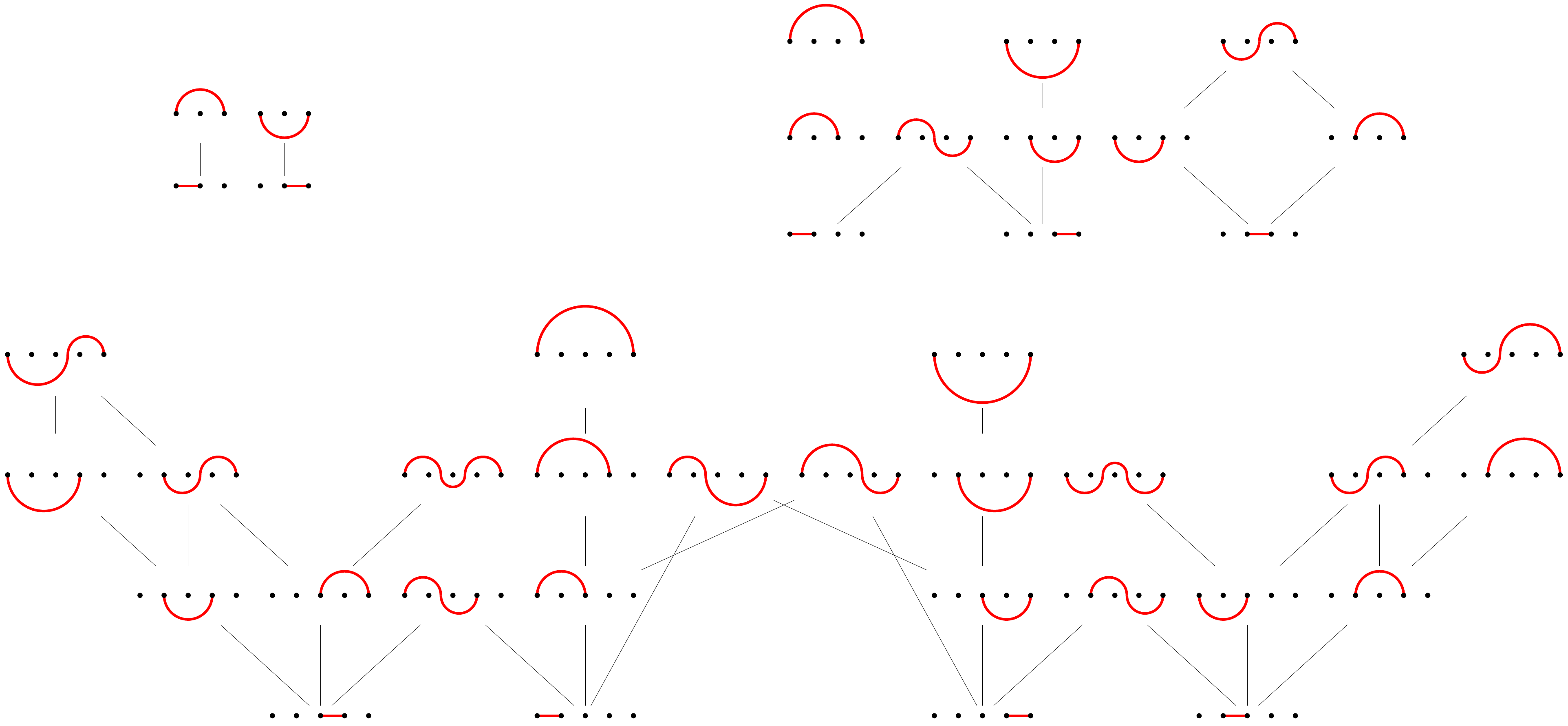}}
	\caption{The weak orders of size $3$ (top left), $4$ (top right), and $5$ (bottom) restricted to their join irreducibles represented by the corresponding arcs.}
	\label{fig:induction35}
\end{figure}
\end{remark}

\begin{remark}
As illustrated in \cref{fig:induction35}, the weak order on join irreducible of~$\f{S}_n$ has interesting enumerative properties. Let us just mention here that it has
\begin{itemize}
\item $2^n - n - 1$ elements (permutations with a single descent, or arcs) \OEIS{A000295},
\item $2^{n+1} - n^2 - n - 2$ cover relations (in bijection with arcs of size $n+1$ crossing the horizontal axis, or with subsets of $[n+1]$ crossing their complement) \OEIS{A324172},
\item $n(n+1)2^{n-2}$ intervals (including the singletons) \OEIS{A001788}.
\end{itemize}
\end{remark}


\subsection{Semi-crossing arc bidiagrams}
\label{subsec:SCABs}

We now describe the canonical complex of the weak order as defined in \cref{subsec:canonicalComplex} in terms of the following combinatorial objects, illustrated in \cref{fig:bidExamples}.

\begin{figure}[b]
	\centerline{\includegraphics[scale=.5]{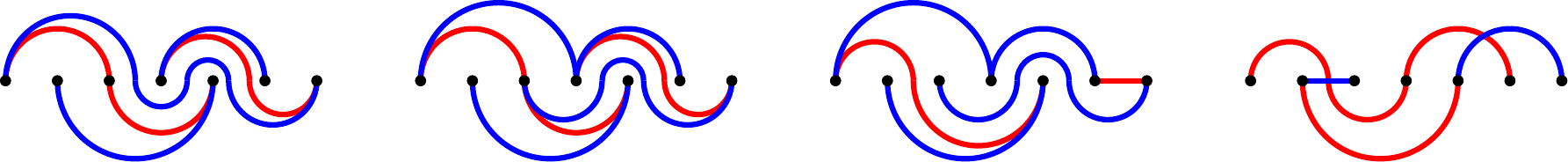}}
	\caption{SCABs of the intervals $[2531746, 2531746], [2531746, 2537146], [2513476, 2537146]$ and $[5264137, 6574231]$.}
	\label{fig:bidExamples}
\end{figure}

\begin{definition}
\label{def:SCAB}
A \defn{semi-crossing arc bidiagram} (or \defn{SCAB} for short) is a disjoint union~$\delta_\join \sqcup \delta_\meet$ of non-crossing arc diagrams such that for any~$\alpha_\join \eqdef (a_\join, b_\join, A_\join, B_\join) \in \delta_\join$ and ${\alpha_\meet \eqdef (a_\meet, b_\meet, A_\meet, B_\meet) \in \delta_\meet}$, there is no~$u < v$ with ${u \in (A_\meet \cup \{a_\meet\}) \cap (B_\join \cup \{a_\join\})}$ and ${v \in (A_\join \cup \{b_\join\}) \cap (B_\meet \cup \{b_\meet\})}$.
The \defn{semi-crossing complex} is the simplicial complex whose ground set contains two copies~$\alpha_\join$ and~$\alpha_\meet$ of each arc~$\alpha$ and whose simplices are all semi-crossing arc bidiagrams.
\end{definition}

\begin{remark}
Visually, a semi-crossing arc bidiagram~$\delta_\join \sqcup \delta_\meet$ is a collection of arcs such that
\begin{itemize}
\item no two arcs of~$\delta_\join$ (resp.~of~$\delta_\meet$) cross in their interiors, or start or end at the same points,
\item no two arcs~$\alpha_\join \in \delta_\join$ and~$\alpha_\meet \in \delta_\meet$ cross in their interiors with~$\alpha_\join$ going up and $\alpha_\meet$ going down at the crossing, or start at the same point with~$\alpha_\join$ leaving above~$\alpha_\meet$, or end at the same point with~$\alpha_\join$ arriving below~$\alpha_\meet$ at this point.
\end{itemize}
\end{remark}

\begin{remark}
Before going further, we report in \cref{table:SCABs} on the number of semi-crossing arc bidiagrams~${\delta_\join \sqcup \delta_\meet}$ according to the cardinalities~$|\delta_\join|$ and~$|\delta_\meet|$ for~$n = 2$ to~$6$.

\begin{table}
	\begin{gather*}
	\begin{array}[t]{l|cc}
		  & 0 & 1 \\
		\hline
		0 & 1 & 1 \\ 
		1 & 1 & 0
	\end{array}
	\qquad
	\begin{array}[t]{l|ccc}
		  0 & 1 & 2 \\
		\hline
		0 & 1 & 4 & 1 \\
		1 & 4 & 6 & 0 \\ 
		2 & 1 & 0 & 0
	\end{array}
	\qquad
	\begin{array}[t]{l|cccc}
		  & 0 & 1 & 2 & 3 \\
		\hline
		0 & 1 & 11 & 11 & 1 \\ 
		1 & 11 & 54 & 24 & 0 \\
		2 & 11 & 24 & 2 & 0 \\ 
		3 & 1 & 0 & 0 & 0
	\end{array}
	\\
	\begin{array}[t]{l|ccccc}
		  & 0 & 1 & 2 & 3 & 4 \\
		\hline
		0 & 1 & 26 & 66 & 26 & 1 \\
		1 & 26 & 300 & 420 & 80 & 0 \\ 
		2 & 66 & 420 & 320 & 20 & 0 \\ 
		3 & 26 & 80 & 20 & 0 & 0 \\ 
		4 & 1 & 0 & 0 & 0 & 0
	\end{array}
	\qquad
	\begin{array}[t]{l|cccccc}
		  & 0 & 1 & 2 & 3 & 4 & 5 \\
		\hline
		0 & 1 & 57 & 302 & 302 & 57 & 1 \\
		1 & 57 & 1340 & 4145 & 2505 & 240 & 0 \\
		2 & 302 & 4145 & 8270 & 3035 & 120 & 0 \\
		3 & 302 & 2505 & 3035 & 562 & 5 & 0 \\ 
		4 & 57 & 240 & 120 & 5 & 0 & 0 \\ 
		5 & 1 & 0 & 0 & 0 & 0 & 0
	\end{array}
	\end{gather*}
	\caption{The number of semi-crossing arc bidiagrams~${\delta_\join \sqcup \delta_\meet}$ according to the cardinalities~$|\delta_\join|$ and~$|\delta_\meet|$ for~$n = 2$ to~$6$.}
	\label{table:SCABs}
\end{table}

In these tables, observe that
\begin{itemize}
\item the first row (resp.~column) corresponds to the intervals~$[\sigma, w_\circ]$ (resp.~$[e, \sigma]$) for the permutations~$\sigma$ of~$[n]$ with $k$ ascents (resp.~descents) and are thus counted by the Eulerian numbers~\OEIS{A008292},
\item the last row (resp.~column) corresponds to the single interval~$[w_\circ, w_\circ]$ (resp.~$[e, e]$).
\end{itemize}
\end{remark}

We now connect the semi-crossing arc bidiagrams with the canonical complex of the weak order using \cref{coro:weakOrderArcs}.

\begin{proposition}
\label{prop:SCAB}
The map~$[\sigma, \tau] \mapsto \b{\delta}_\join(\sigma) \sqcup \b{\delta}_\meet(\tau)$ is a bijection between the intervals of the weak order on~$\f{S}_n$ and the semi-crossing arc bidiagrams.
Hence, the canonical complex of the weak order is isomorphic to the semi-crossing complex.
\end{proposition}

\begin{proof}
By \cref{prop:NCAD}, the maps~$\sigma \mapsto \b{\delta}_\join(\sigma)$ and~$\tau \mapsto \b{\delta}_\meet(\tau)$ are both bijections from permutations to non-crossing arc diagrams.
Moreover, $\sigma \le \tau$ if and only if each canonical joinand of~$\sigma$ is smaller than each canonical meetand of~$\tau$, which is equivalent to each arc of~$\b{\delta}_\join(\sigma)$ being semi-crossing each arc of~$\b{\delta}_\meet(\tau)$ by \cref{coro:weakOrderArcs}\,(iii).
Hence, $[\sigma, \tau] \mapsto \b{\delta}_\join(\sigma) \sqcup \b{\delta}_\meet(\tau)$ is a bijection from intervals to semi-crossing arc bidiagrams.
Finally, $\b{\delta}_\join(\sigma) \sqcup \b{\delta}_\meet(\tau)$ corresponds to the canonical representation of~$[\sigma, \tau]$ since~$\b{\delta}_\join(\sigma)$ corresponds to the canonical join representation of~$\sigma$ and~$\b{\delta}_\meet(\tau)$ corresponds to the canonical meet representation of~$\tau$.
\end{proof}

\begin{figure}[b]
	\centerline{\includegraphics[scale=1]{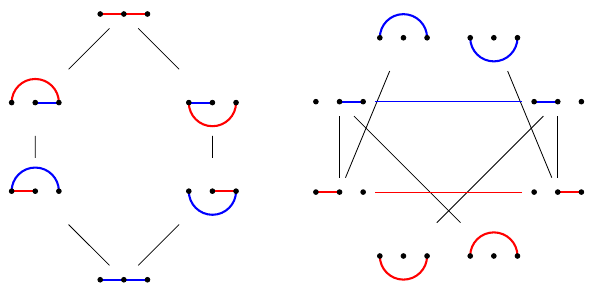}}
	\caption{The weak order on~$\f{S}_3$ with permutations labeled by semi-crossing arc bidiagrams, and the canonical complex of $\mathfrak{S}_3$ with join and meet irreducible permutations labeled by arcs.}
	\label{fig:CCS3}
\end{figure}

For instance, the canonical complexes of the weak orders on~$\f{S}_3$ and~$\f{S}_4$ are illustrated in \cref{fig:CCS3,fig:CCS4}.
As usual, since the canonical complex is flag by \cref{prop:canonicalComplexFlag}, we only represent its graph.
The central symmetry corresponds to the map~$\kappa$ of \cref{rem:symmetryKappa}, which just corresponds to the exchange of color of the arcs by~\cref{prop:kappaChangesColor}.

We now characterize the semi-crossing arc bidiagrams corresponding to singleton intervals.
As illustrated in \cref{fig:diagExamples}, these semi-crossing arc bidiagrams are certain paths.
We thus define the \defn{source}~$s(\alpha)$ and the \defn{target}~$t(\alpha)$ of an arc~$\alpha \eqdef (a,b,A,B)$ in a semi-crossing arc bidiagram~$\delta_\join \sqcup \delta_\meet$ as~$s(\alpha) = b$ and~$t(\alpha) = a$ if~$\alpha \in \delta_\join$ and~$s(\alpha) = a$, and~$t(\alpha) = b$ if~$\alpha \in \delta_\meet$.

\begin{proposition}
\label{prop:singletonIntervals}
The following conditions are equivalent for a semi-crossing arc bidiagram~$\delta_\join \sqcup \delta_\meet$:
\begin{enumerate}[(i)]
\item $\delta_\join \sqcup \delta_\meet = \b{\delta}_\join(\sigma) \sqcup \b{\delta}_\meet(\sigma)$ for a permutation~$\sigma \in \f{S}_n$,
\item there is an labeling~$\alpha_1 = (a_1, b_1, A_1, B_1), \dots, \alpha_{n-1} = (a_{n-1}, b_{n-1}, A_{n-1}, B_{n-1})$ of~$\delta_\join \sqcup \delta_\meet$ such that~$t(\alpha_i) = s(\alpha_{i+1})$ and~$a_i \notin B_j$ for~$1 \le i < j \le n-1$.
\end{enumerate}
If these conditions hold, then the arcs of~$\delta_\join \sqcup \delta_\meet$ do not contain crossings in their interiors.
\end{proposition}

\begin{proof}
For (i) $\Rightarrow$ (ii), the arc~$\alpha_i$ is the arc~$\b{\alpha}_\join(\sigma, i)$ if~$\sigma_i > \sigma_{i+1}$ and $\b{\alpha}_\meet(\sigma, i)$ if~$\sigma_i < \sigma_{i+1}$ described before \cref{prop:NCAD}.
For (ii) $\Rightarrow$ (i), the permutation~$\sigma$ is given by~$[s(\alpha_1), t(\alpha_1), \dots, t(\alpha_{n-1})]$.
\end{proof}

Finally, let us insist again here that this combinatorial model for the intervals of the weak order is adapted to the study of its quotients.
The next statement follows from \cref{prop:kappaChangesColor,prop:subarcs}.

\begin{proposition}
For any lower ideal~$I$ of the subarc order, the canonical complex of the quotient of the weak order by~$\equiv_I$ is isomorphic to the subcomplex of the semi-crossing complex induced by~$\set{\alpha_\join}{\alpha \in I} \sqcup \set{\alpha_\meet}{\alpha \in I}$.
\end{proposition}

We conclude with a conjecture motivated by \cref{prop:singletonIntervals} and checked by computer experiments for all lattice quotients of the weak order on~$\f{S}_n$ for~$n \le 5$.

\begin{conjecture}
The semi-crossing arc bidiagram corresponding to an inclusion minimal interval in a lattice quotient of the weak order does not contain any crossing in the interior of its arcs.
\end{conjecture}


\subsection{Weak order interval posets}
\label{subsec:WOIPs}

Following a classical result of A.~Bj\"orner and M.~Wachs~\cite[Thm.~6.8]{BjornerWachs}, G.~Ch\^atel, V.~Pilaud and V.~Pons already studied in~\cite{ChatelPilaudPons} a family of posets in bijection with the intervals of the weak order.
For a poset~$\less$ on~$[n]$ and~$1 \le u < v \le n$, we say that the relation~$u \less v$ is \defn{increasing} and that the relation~$v \less u$ is \defn{decreasing} (and we write~$u \more v$ for decreasing relations).

\begin{proposition}[{\cite[Thm.~6.8]{BjornerWachs} \& \cite[Prop.~26]{ChatelPilaudPons}}]
\label{prop:WOIP}
The following conditions are equivalent for a poset~$\less$ on~$[n]$:
\begin{itemize}
\item the linear extensions of $\less$ form an interval~$[\sigma, \tau]$ of the weak order,
\item there are~$\sigma \le \tau$ in the weak order such that the increasing relations of~$\less$ are the non-inversions of~$\tau$ and the decreasing relations are the inversions of~$\sigma$,
\item $a \less c$ implies~$a \less b$ or $b \less c$, and $a \more c$ implies $a \more b$ or $b \more c$ for all~$1 \le a < b < c \le n$.
\end{itemize}
Such a poset is called a \defn{weak order interval poset} (or \defn{WOIP} for short).
\end{proposition}

Although they are specific to the weak order and do not behave well with respect to its quotients, the WOIPs are combinatorial objects in bijection with intervals of the weak order, and thus with SCABs.
It is thus relevant to provide direct explicit bijections between~SCABs and~WOIPs, which are illustrated in \cref{fig:WOIP_SCAB}.
We need the following definitions, illustrated in \cref{fig:WOIP_SCAB}.

\begin{figure}[b]
	\centerline{\includegraphics[scale=.5]{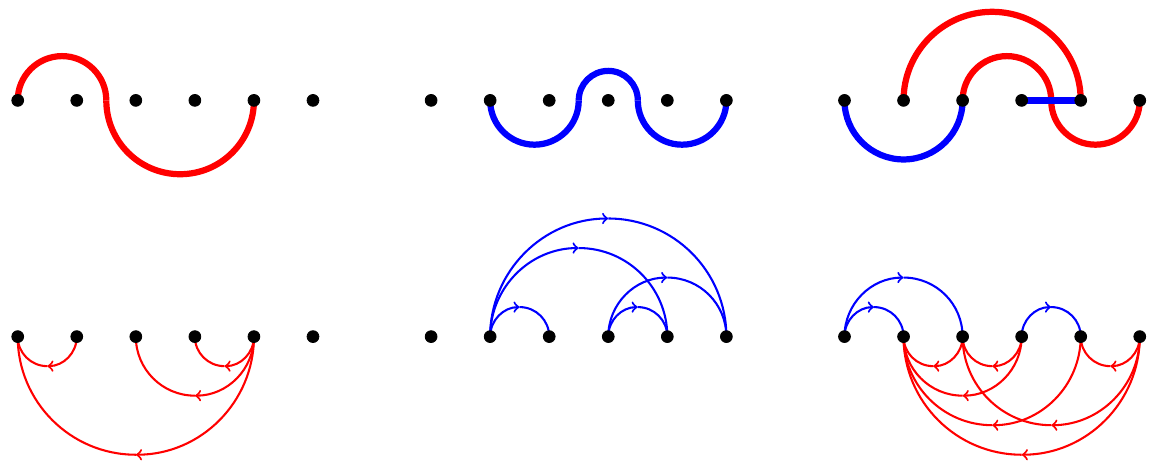}}
	\caption{The bijection between SCABs (top) and WOIPs (bottom). We represent a WOIP $\less$ with arcs joining $a<b$ from above (resp.~below) in blue (resp.~red) when $a\less b$ (resp.~$a\more b$).}
	\label{fig:WOIP_SCAB}
\end{figure}

\medskip\noindent
\textsf{From SCABs to WOIPs.}
For any arc~$\alpha \eqdef (a, b, A, B)$, we denote by~$\b{\less}_\alpha$ the set of increasing relations~$u \; {\b{\less}_\alpha} \, v$ (by~$\b{\more}_\alpha$ the set of decreasing relations~$u \; {\b{\more}_\alpha} \, v$) where~${u < v}$ with~$u \in A \cup \{a\}$ and~$v \in B \cup \{b\}$.
For a NCAD~$\delta$, we denote by~$\b{\less}_\delta$ (resp.~$\b{\more}_\delta$) the transitive closure of the union of the increasing relations~$\set{\b{\less}_{\alpha_\join}}{\alpha_\join \in \delta_\join}$ (resp.~of the decreasing relations~$\set{\b{\more}_{\alpha_\meet}}{\alpha_\meet \in \delta_\meet}$).
Finally, to a SCAB~$\delta_\join \sqcup \delta_\meet$ we associate the WOIP ${\b{\more}_{\delta_\join}} \sqcup {\b{\less}_{\delta_\meet}}$ (see \cref{prop:SCAB-WOIP}).

\medskip\noindent
\textsf{From WOIPs to SCABs.}
Fix a WOIP~${\more} \sqcup {\less}$ where~$\more$ denote the decreasing relations and $\less$ denote the increasing relations.
We say that an increasing (resp.~decreasing) cover relation~$a \lesscover b$ (resp.~$a \morecover b$) is \defn{maximal} if there is no cover relation~$a' \lesscover b$ (resp.~$a' \morecover b$) with~${a' < a}$ or~$a \lesscover b'$ (resp.~$a \morecover b'$) with~${b < b'}$.
To a maximal decreasing (resp.~increasing) cover relation~$a \morecover b$ (resp.~$a \lesscover b$), we associate the arc~$\b{\alpha}(a \morecover b) \eqdef (a, b, {]a,b[} \cap \downIdeal[a], {]a,b[} \cap \upIdeal[b])$ (resp.~${\b{\alpha}(a \lesscover b) \eqdef (a, b, {]a,b[} \cap \downIdeal[b], {]a,b[} \cap \upIdeal[a])}$), where~$\downIdeal[x]$ and~$\upIdeal[x]$ denote the lower and upper ideals of~$\less$ generated by~$x$.
We denote by~$\b{\delta}(\more)$ (resp.~$\b{\delta}(\less)$) the set of arcs~$\b{\alpha}(a \morecover b)$ (resp.~$\b{\alpha}(a \lesscover b)$) for all maximal decreasing (resp.~increasing) cover relations of~$\less$.
Finally, to the WOIP~${\more} \sqcup {\less}$, we associate the SCAB~$\b{\delta}(\more) \sqcup \b{\delta}(\less)$ (see \cref{prop:SCAB-WOIP}).

\begin{proposition}
\label{prop:SCAB-WOIP}
The maps~$\delta_\join \sqcup \delta_\meet \mapsto {\b{\more}_{\delta_\join}} \sqcup {\b{\less}_{\delta_\meet}}$ and~${\more} \sqcup {\less} \mapsto \b{\delta}(\more) \sqcup \b{\delta}(\less)$ are inverse bijections between SCABs and WOIPs.
\end{proposition}

\begin{proof}
Consider an interval~$[\sigma, \tau]$ of the weak order corresponding to a SCAB~$\delta_\join \sqcup \delta_\meet$.
For any arc~$\alpha_\join$, the decreasing relations of~$\b{\more}_{\alpha_\join}$ are precisely the inversions of~$\b{\sigma}_\join(\alpha_\join)$ by \cref{lem:arcInversions}.
Hence, the decreasing relations of~$\b{\more}_{\delta_\join}$ are precisely the inversions of~$\sigma = \bigJoin \set{\b{\sigma}_\join(\alpha_\join)}{\alpha_\join \in \delta_\join}$ (since the inversion set of a join is the transitive closure of the union of the inversion sets of the joinands).
Similarly, the increasing relations of~$\b{\less}_{\delta_\meet}$ are precisely the non-inversions of~$\tau = \bigMeet \set{\b{\sigma}_\meet(\alpha_\meet)}{\alpha_\meet \in \delta_\meet}$.
Hence~${\b{\more}_{\delta_\join}} \sqcup {\b{\less}_{\delta_\meet}}$ is indeed the WOIP of the interval~$[\sigma, \tau]$ by \cref{prop:WOIP}\,(ii).
Finally, to see that the two maps are inverse to each other, we just need to observe that the relations~$a \, {\b{\less}_{\delta_\meet}} \, b$ created out of the extremities of the arcs~$(a, b, A, B)$ of~$\delta_\meet$ are precisely the maximal cover relations of~$\b{\less}_{\delta_\meet}$ (and similarly for~$\b{\more}_{\delta_\join}$). 
\end{proof}


\subsection{Kreweras maps in quotients of the weak order}
\label{subsec:KrewerasBidiagrams}

We finally describe the Kreweras maps defined in~\cref{subsubsec:Kreweras} in all quotients of the weak order in terms of semi-crossing arc bidiagrams.
For this, we first connect the canonical join representation of a permutation to the canonical join representation of the minimal element in its class for a given congruence, as illustrated in \cref{fig:Kreweras}.
In the following proposition, we call weak order on arcs the order~$\alpha \le \alpha'$ if $\b{\sigma}_\join(\alpha) \le \b{\sigma}_\join(\alpha')$ (see \cref{subsec:weakOrderArcs,fig:induction35}).

\begin{proposition}
\label{prop:Kreweras}
Consider an upper ideal~$I$ of the subarc order and a permutation~$\sigma$.
Let~$X$ be the intersection of~$I$ with the lower ideal generated by the non-crossing arc diagram~$\b{\delta}_\join(\sigma)$ in the weak order on arcs.
Let~$Y$ be the set of arcs~$(a, b, A, B)$ of~$X$ such that there is~$a < p < b$ such that both arcs~$(a, p, A \cap {]a,p[}, B \cap {]a,p[})$ and~$(p, b, A \cap {]p,b[}, B \cap {]p,b[})$ belong to~$X$.
Then the non-crossing arc diagram~$\b{\delta}_\join \big( \projDown[\equiv_I](\sigma) \big)$ is the set of maximal elements of~$X \ssm Y$ in the weak order on arcs.
\end{proposition}

\begin{proof}
By \cref{prop:canonicalJoinRepresentationsQuotient,prop:canonicalJoinRepresentationProjection}, $\b{\delta}_\join \big( \projDown[\equiv_I](\sigma) \big)$ is a non-crossing arc diagram contained in~$X$.
Among all options, we need to choose the non-crossing arc diagram with maximal join.
This implies that the arcs of~$Y$ cannot appear in $\b{\delta}_\join \big( \projDown[\equiv_I](\sigma) \big)$ since
\[
\b{\sigma}_\join(a, b, A, B) \le \b{\sigma}_\join(a, p, A \cap {]a,p[}, B \cap {]a,p[}) \join \b{\sigma}_\join(p, b, A \cap {]p,b[}, B \cap {]p,b[})
\]
for any arc~$(a, b, A, B)$ and any~$a < p < b$ by \cref{lem:arcInversions}.
We finally claim that any two incomparable elements in~$X \ssm Y$ cannot cross.
Hence, the maximal elements of~$X \ssm Y$ form a non-crossing arc diagram, which must therefore be~$\b{\delta}_\join \big( \projDown[\equiv_I](\sigma) \big)$.
To prove the claim, consider two arcs~$\alpha \eqdef (a, b, A, B)$ and~$\alpha' \eqdef (a', b', A', B')$ in~$X$ which are incomparable and crossing.
Assume for instance that there are~$u < v$ such that~${u \in (A \cup \{a\}) \cap (B' \cup \{a'\})}$ and ${v \in (A' \cup \{b'\}) \cap (B \cup \{b\})}$.
We can moreover assume that~$\alpha$ and~$\alpha'$ agree on~$]u,v[$, meaning that~$A \cap {]u,v[} = A' \cap {]u,v[}$ and~$B \cap {]u,v[} = B' \cap {]u,v[}$.
Since $\alpha$ and~$\alpha'$ are incomparable, this implies that~$u \ne a$ or~$v \ne b$.
Moreover, if~$u \ne a$ and~$v \ne b$, then there is a crossing between any arc larger than~$\alpha$ and any arc larger than~$\alpha'$ in the weak order on arcs, so that~$\alpha$ and~$\alpha'$ cannot both belong to~$X$.
We conclude that either~$u \ne a$ or~$v \ne b$, so that precisely one of the arcs
\[
(a, u, A \cap {]a, u[}, B \cap {]a, u[})
\quad\text{and}\quad
(v, b, A \cap {]v, b[}, B \cap {]v, b[})
\]
is non-trivial (not reduced to a single point).
Moreover, this arc belongs to~$X$ since~$\alpha$ does, and the arc~$(u, v, A \cap {]u, v[}, B \cap {]u, v[}) = (u, v, A' \cap {]u, v[}, B' \cap {]u, v[})$ belongs to~$X$ since~$\alpha'$ does.
This implies that~$\alpha$ is in~$Y$ as it can be decomposed into exactly two subarcs that belong to~$X$.
\end{proof}

This enables us to compute the Kreweras maps in quotients of the weak order directly on non-crossing arc diagrams.
For this, let us extend the notations of~\cref{subsubsec:Kreweras} to quotients and transport them to non-crossing arc diagrams.
For an upper ideal~$I$ of the subarc order, each equivalence class of~$\equiv_I$ is an interval~$[x,y]$ of the weak order and thus corresponds to two non-crossing arc diagrams~$\delta_\join \eqdef \b{\delta}_\join(x)$ and~$\delta_\meet \eqdef \b{\delta}_\meet(y)$.
We denote by~$\eta^I_\join$ and~$\eta^I_\meet$ the two opposite maps defined by~$\eta^I_\join(\delta_\meet) = \delta_\join$ and~$\eta^I_\meet(\delta_\join) = \delta_\meet$.
We just write~$\eta_\join$ and~$\eta_\meet$ when~$I$ is the set of all arcs.
Note that~$\eta_\join = \b{\delta}_\join \circ \b{\delta}_\meet^{-1}$ and~$\eta_\meet = \b{\delta}_\meet \circ \b{\delta}_\join^{-1}$ are easily computed from the descriptions of the maps~$\b{\delta}_\join$ and $\b{\delta}_\meet$ (see \cref{subsec:NCADs}) and of their inverses (see the explicit description in~\cite{Reading-arcDiagrams}).
\cref{prop:Kreweras} enables to compute~$\eta^I_\join$ and~$\eta^I_\meet$  in general.

\begin{corollary}
\label{coro:Kreweras}
Consider an upper ideal~$I$ of the subarc order and a non-crossing arc diagram~$\delta_\meet$ with all arcs in~$I$.
Then the non-crossing arc diagram~$\eta^I_\join(\delta_\meet)$ is obtained from~$\eta_\join(\delta_\meet)$ by applying the algorithm of \cref{prop:Kreweras}.
\end{corollary}

\begin{example}
\label{exm:KrewerasComplement}
When~$I$ is the upper ideal of up arcs corresponding to the sylvester congruence, the description of \cref{coro:Kreweras} can be translated to the classical description of the Kreweras complement of a non-crossing partition.
Namely, the Kreweras complement of a non-crossing partition is obtained by shifting the points and connecting the points in the same connected component.
See \cref{fig:KrewerasComplement}.

\begin{figure}[h]
	\centerline{\includegraphics[scale=.5]{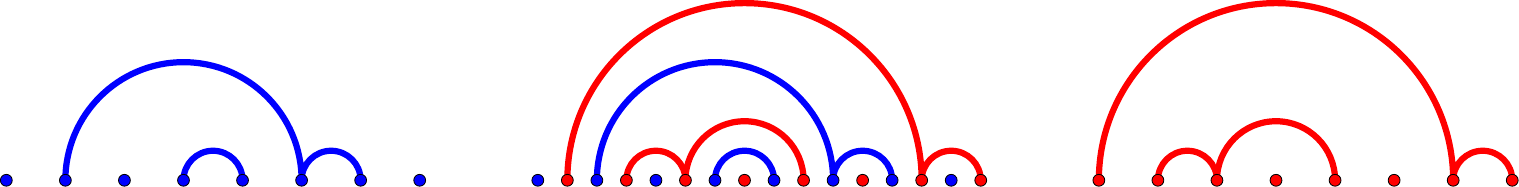}}
	\caption{Classical Kreweras complement on non-crossing partitions.}
	\label{fig:KrewerasComplement}
\end{figure}
\end{example}

\begin{figure}
	\centerline{\includegraphics[scale=.9]{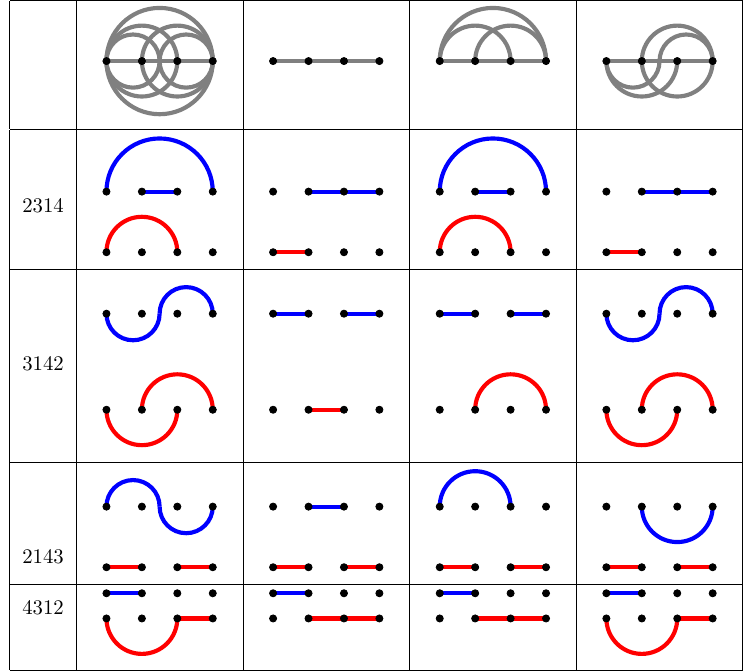}}
	\caption{Illustrations of computations of~$\b{\delta}_\join \big( \projDown[\equiv_I](\sigma) \big)$ and $\eta^I_\join(\delta_\meet)$ in the trivial congruence, the descent congruence, the sylvester congruence, and a generic congruence.}
	\label{fig:Kreweras}
\end{figure}


\addtocontents{toc}{ \vspace{.1cm} }
\section*{Acknowledgements}

We are grateful to J.-C. Novelli for many discussions and suggestions on the content and presentation of this paper, to S.~Giraudo for a careful proofreading of a preliminary version, and to an anonymous referee for helpful comments.

%
%


\bibliographystyle{alpha}
\bibliography{canonicalComplex}
\label{sec:biblio}

\end{document}